\documentclass[hidelinks,nohypdvips,onefignum,onetabnum]{siamart251216}
\usepackage{amsfonts,amsmath,amssymb,graphicx,booktabs,array,placeins,needspace}
\usepackage{algpseudocode}
\usepackage[nopatch=eqnum]{microtype}
\newsiamremark{assumption}{Assumption}
\newsiamremark{remark}{Remark}
\crefname{assumption}{Assumption}{Assumptions}
\newcommand{\R}{\mathbb R}
\newcommand{\C}{\mathbb C}
\newcommand{\T}{\mathbb T}
\newcommand{\E}{\mathcal E}

\newcommand{\HH}{\mathcal H}
\newcommand{\PP}{\mathcal P}
\newcommand{\LL}{\mathcal L}

\newcommand{\norm}[1]{\left\|#1\right\|}
\newcommand{\ip}[2]{\left\langle#1,#2\right\rangle}
\newcommand{\re}{\operatorname{Re}}

\newcommand{\dd}{\,\mathrm d}
\newcommand{\Tr}{\mathcal T}
\newcommand{\eps}{\varepsilon}

\headers{Energy-constrained implicit SAT and robust error estimates}{Hao Hu and Haijun Yu}
\title{An energy-constrained implicit step-and-truncate algorithm and its robust error estimates}
\author{Hao Hu\thanks{School of Mathematical Sciences, University of Chinese Academy of Sciences, Beijing 100049, China, and LSEC, Institute of Computational Mathematics and Scientific/Engineering Computing, Academy of Mathematics and Systems Science, Beijing 100190, China (\email{huhao@lsec.cc.ac.cn}).}
\and Haijun Yu\thanks{Corresponding author. State Key Laboratory of Mathematical Sciences \& LSEC, Institute of Computational Mathematics and Scientific/Engineering Computing, Academy of Mathematics and Systems Science, Beijing 100190, China, and School of Mathematical Sciences, University of Chinese Academy of Sciences, Beijing 100049, China (\email{hyu@lsec.cc.ac.cn}).}}
\hypersetup{pdftitle={An energy-constrained implicit step-and-truncate algorithm and its robust error estimates},pdfauthor={Hao Hu and Haijun Yu}}
\begin{document}
\maketitle
\begin{abstract}
Deriving mesh-robust error bounds is an open problem for low-rank methods, including dynamical low-rank (DLR) and step-and-truncate (SAT) approaches, as highlighted in a recent review \cite[Section~8]{Review2025}.
Existing bounds can depend on Lipschitz constants that grow with inverse mesh size, making them impractical, particularly for kinetic equations, a major application of low-rank methods.
We develop a unified framework for implicit SAT that provides a discrete counterpart of energy-localization analysis for nonlinear partial differential equations.
By restricting numerical solutions to energy-localized regions, we obtain one-sided Lipschitz stability with constants independent of spatial resolution.
The resulting error estimates and time-step restrictions are mesh robust, answering this open problem for the class of implicit SAT schemes covered by our assumptions.
We emphasize that energy localization provides a general mechanism for mesh-independent robust error estimates in low-rank approximation, extending beyond the SAT methods considered here. Analogous results for DLR methods will be presented in subsequent work.
Returning to implicit SAT, we demonstrate the broad applicability of the framework through error estimates for different classes of equations. Specifically, for semilinear parabolic and Vlasov--Poisson equations, we establish robust first-order error bounds for backward-Euler SAT under suitable regularity assumptions.
Energy localization also has practical implications.
We design an energy-constrained implicit SAT algorithm.
Implicit SAT methods without energy constraints can select incorrect solution branches even when their algebraic residuals are small.
Moreover, time-step reduction can be essential for continuation into larger energy regions.
Numerical experiments support these conclusions.
\end{abstract}
\section{Introduction}\label{sec:intro}
Spatial discretization of time-dependent partial differential equations can produce large matrix differential equations with stiff vector fields. Low-rank approximation reduces their storage and computational cost by exploiting separation between groups of variables. A rank-$r$ factorization of an $m\times n$ matrix uses $O(r(m+n))$ entries instead of $mn$. Evolving this representation efficiently while preserving the stability of the underlying equation is a central problem in low-rank time integration; see \cite{Bachmayr2023,Review2025} for broader accounts.

Two main approaches have been developed. Dynamical low-rank (DLR) approximation projects the vector field onto the tangent space of a fixed-rank manifold \cite{KochLubich2007}. The projector-splitting integrator of Lubich and Oseledets \cite{LubichOseledets2014} avoids the small-singular-value restrictions of direct factor integration, with robustness established by Kieri, Lubich, and Walach \cite{KieriLubichWalach2016}. Further developments include low-rank splitting for stiff semilinear equations \cite{OstermannPiazzolaWalach2019}, basis-update and Galerkin integration \cite{CerutiLubich2022}, and adaptive ranks \cite{CerutiKuschLubich2022}. The step-and-truncate (SAT) approach instead advances the solution and then compresses the result. Rodgers, Dektor, and Venturi \cite{RodgersDektorVenturi2022} formulated and analyzed adaptive time integration of this type on tensor manifolds.

\Needspace{4\baselineskip}
Low-rank methods are widely used in the numerical solution of partial differential equations. Meyer, Manthe, and Cederbaum \cite{MeyerMantheCederbaum1990} introduced the multiconfiguration time-dependent Hartree approach for the time-dependent Schr\"odinger equation. For kinetic equations in three spatial and three velocity dimensions, the six-dimensional phase space makes full-grid discretizations expensive. For the Vlasov--Poisson equation, Kormann \cite{Kormann2015} developed a semi-Lagrangian solver in tensor-train format, while Einkemmer and Lubich \cite{EinkemmerLubich2018} constructed a low-rank projector-splitting integrator. Guo and Qiu \cite{GuoQiu2022} developed adaptive low-rank tensor approximations for linear transport and nonlinear Vlasov equations, and subsequently introduced conservative truncation for Vlasov dynamics \cite{GuoQiu2024}. For radiation transport, Peng, McClarren, and Frank \cite{PengMcClarrenFrank2020} combined dynamical low-rank splitting with finite-volume spatial discretization and a spherical-harmonics angular expansion.

For stiff spatial discretizations, implicit SAT is particularly attractive because implicit time stepping can alleviate the stability restrictions of explicit methods. Rodgers and Venturi \cite{RodgersVenturi2023} developed rank-adaptive implicit step-truncation methods. More recent implicit low-rank methods include Krylov-based integrators \cite{ElKahzaEtAl2024}, the adaptive matrix schemes of Appel\"o and Cheng \cite{AppeloCheng2025}, and reduced augmentation implicit low-rank methods \cite{NakaoQiuEinkemmer2025}. Stability and efficiency benefits are demonstrated, for example, by the Allen--Cahn comparisons in \cite[Figure~1 and Table~1]{RodgersVenturi2023}, which report lower computational cost for implicit SAT at comparable accuracy.

Despite substantial progress, low-rank methods remain a young field with many open questions. A prominent gap is the lack of a general theory of mesh-robust error bounds \cite[Section~8]{Review2025}. Existing estimates depend on Lipschitz constants that may grow under spatial refinement. In the analysis of implicit SAT, this dependence introduces an additional difficulty. Writing the full Lipschitz constant of the discretized vector field as $L_{\mathrm{Lip},h}$, the nonlinear stability argument in \cite[Section~4.1, equation~(36)]{RodgersVenturi2023} requires $\tau L_{\mathrm{Lip},h}\le1/2$. Its appearance therefore affects both the usefulness of the error bound and the theoretical justification of the advantage of implicit time stepping.

We address these difficulties through energy localization, a classical technique for establishing stability in nonlinear PDEs; see \cite{Kato1975} and \cite[Chapters~15--16]{Taylor2011}.
We develop a discrete counterpart of this approach for implicit SAT.
For a family of spatial discretizations, we choose suitable one-sided energies that are nonincreasing under SVD truncation and verify three conditions:
\begin{enumerate}
\item a mesh-uniform bound on the initial energy;
\item mesh-uniform controlled energy propagation within the energy-localized region;
\item mesh-uniform one-sided Lipschitz stability within that region.
\end{enumerate}
Under the associated mesh-independent step-size and residual conditions, these properties yield existence and uniqueness of the implicit stage in the energy region for each prescribed admissible residual, together with mesh-robust error estimates for the SAT approximations.

We verify the framework for semilinear reaction--diffusion equations $u_t=\kappa\Delta u+f(u)$.
For this class of equations, Ostermann, Piazzola, and Walach \cite{OstermannPiazzolaWalach2019} developed a low-rank Lie--Trotter method that separates the stiff linear and nonstiff nonlinear parts.
It uses the exact matrix-exponential representation of the stiff linear flow for efficient integration and approximates the nonstiff part by dynamical low-rank evolution.
Their error estimate contains a temporal contribution of order $\tau(1+|\log\tau|)$.
Our implicit SAT method requires no separate treatment of the stiff part, and our analysis gives a mesh-robust first-order error bound.

We also verify the framework for the original Vlasov--Poisson system on $\T\times\R$.
Applying our theory gives an $O(\tau+\delta_{K,N})$ error bound with constants independent of spatial resolution, where $\delta_{K,N}$ measures the spatial discretization error.
This provides a mesh-robust estimate for a kinetic equation of particular interest in low-rank approximation and removes the mesh-dependent Lipschitz factors that make existing bounds impractical in this setting.

Energy localization also has practical implications.
We design an energy-constrained implicit SAT algorithm.
Implicit SAT methods without energy constraints can select incorrect solution branches even when their algebraic residuals are small.
Moreover, time-step reduction can be essential for continuation into larger energy regions.
Numerical experiments support these conclusions.

Section~\ref{sec:sat} introduces SAT and the full Lipschitz difficulty. Section~\ref{sec:pde-localization} recalls the energy-localization argument, and Section~\ref{sec:framework} identifies energies preserved by SVD truncation. Section~\ref{sec:stability} establishes existence, uniqueness, and robust error bounds. Sections~\ref{sec:parabolic} and~\ref{sec:vp} verify the assumptions and first-order convergence for semilinear reaction--diffusion and Vlasov--Poisson equations, respectively. Section~\ref{sec:numerics} gives an energy-constrained algorithm with a one-sided stability step restriction and tests it numerically; Section~\ref{sec:conclusion} concludes.

\section{Step-and-truncate approximation and the Lipschitz difficulty}\label{sec:sat}
We begin with the evolution problem, its low-rank representation, and the SAT update. This makes it possible to identify precisely how a full Lipschitz constant, which can grow under mesh refinement, enters both the error estimates and the time-step restrictions.

\subsection{Evolution equations and low-rank representations}
Consider a time-dependent PDE written in the abstract form
\begin{equation}\label{eq:pde-general}
 \partial_t u=\mathcal F(t,u),\qquad u(0)=u_0.
\end{equation}
A tensor-product spatial discretization gives the matrix differential equation
\begin{equation}\label{eq:ode}
 X_h'(t)=F_h(t,X_h(t)),\qquad X_h(0)=X_{h,0}.
\end{equation}
Here $X_h(t)$ is an $m_h\times n_h$ matrix, and $h$ denotes all spatial discretization parameters. We equip the matrix space $H_h$ with the Frobenius inner product $\ip{\cdot}{\cdot}_h$ and its induced norm $\|\cdot\|_h$.

A rank-$r$ approximation has the form
\begin{equation}\label{eq:lowrank-representation}
 Y=USV^*,\qquad U^*U=V^*V=I_r,
\end{equation}
where $U\in\C^{m_h\times r}$ and $V\in\C^{n_h\times r}$ contain the two sets of basis vectors, and $S\in\C^{r\times r}$ contains their coefficients. In the underlying function representation, this corresponds to
$u_r(t,x,v)=\sum_{i,j=1}^r S_{ij}(t)a_i(t,x)\overline{b_j(t,v)}$.
The bases, coefficients, and rank may all change in time. Evolving this representation efficiently is the purpose of the low-rank method.

\subsection{Time stepping followed by truncation}
Let $0=t_0<\cdots<t_N=T$ and $\tau_n=t_{n+1}-t_n$. Given a low-rank matrix $Y^n$, SAT, which stands for step and truncate, first advances the solution and then compresses the result \cite{RodgersDektorVenturi2022}. In the step stage, we use implicit time integration. For simplicity, we focus on backward Euler and allow an inexact solution of its nonlinear equation:
\begin{equation}\label{eq:stage}
 \widehat Y^{n+1}-Y^n-\tau_nF_h(t_{n+1},\widehat Y^{n+1})
     =r_{n+1},
\end{equation}
where $r_{n+1}$ is the full algebraic residual.

Time stepping can increase the rank, so the second step truncates $\widehat Y^{n+1}$ to retain a low-rank representation. For $Z=\sum_{\ell=1}^{q}\sigma_\ell u_\ell v_\ell^*$ with decreasing singular values and a tolerance $\theta\ge0$, typically chosen in relation to the time step $\tau$, the truncation is defined by
\begin{equation}\label{eq:sat-truncation-definition}
 \Tr_\theta Z=\sum_{\ell=1}^{r}\sigma_\ell u_\ell v_\ell^*,
 \qquad
 \|Z-\Tr_\theta Z\|_h
   =\left(\sum_{\ell>r}\sigma_\ell^2\right)^{1/2}\le\theta.
\end{equation}
We choose $r$ as the smallest rank satisfying the tolerance. The retained singular vectors and values supply the factors of $Y^{n+1}$ used at the next time step:
\begin{equation}\label{eq:truncate}
 Y^{n+1}=\Tr_{\theta_n}\widehat Y^{n+1},\qquad
 \|Y^{n+1}-\widehat Y^{n+1}\|_h\le\theta_n.
\end{equation}

\subsection{Why the Lipschitz constant causes difficulties for error bounds and time-step restrictions}
Let $X^n\in H_h$ be a prescribed reference sequence for the SAT iterates. The total error satisfies
\[
 \|Y^n-X_h(t_n)\|_h
 \le \|Y^n-X^n\|_h+\|X^n-X_h(t_n)\|_h.
\]
We call the reference values $X^n$ comparison states. Taking $X^n=X_h(t_n)$ makes the second term vanish. A regular low-rank approximation is another possible choice \cite{KochLubich2007,KieriLubichWalach2016}; it separates the approximation error from the error in reproducing the reference trajectory. To estimate the first term, rewrite \eqref{eq:stage} as
\begin{equation}\label{eq:sat-comparison-identity}
 \begin{aligned}
 \widehat Y^{n+1}-X^{n+1}
 ={}&Y^n-X^n+r_{n+1}-d_{n+1}\\
 &+\tau_n\bigl(F_h(t_{n+1},\widehat Y^{n+1})
                -F_h(t_{n+1},X^{n+1})\bigr),
 \end{aligned}
\end{equation}
where
\begin{equation}\label{eq:defect}
 d_{n+1}=X^{n+1}-X^n-\tau_nF_h(t_{n+1},X^{n+1}).
\end{equation}
We call $d_{n+1}$ the discrete defects of the comparison states.
The difference of vector fields is therefore unavoidable. A direct norm estimate uses
\begin{equation}\label{eq:full-lipschitz}
 \|F_h(t,U)-F_h(t,V)\|_h\le L_{\mathrm{Lip},h}\|U-V\|_h.
\end{equation}
Set $e_n=\|Y^n-X^n\|_h$ and $\widehat e_{n+1}=\|\widehat Y^{n+1}-X^{n+1}\|_h$. If $\tau_nL_{\mathrm{Lip},h}<1$, the triangle inequality yields
\begin{equation}\label{eq:full-lipschitz-recursion}
 \widehat e_{n+1}\le
 \frac{e_n+\|r_{n+1}\|_h+\|d_{n+1}\|_h}{1-\tau_nL_{\mathrm{Lip},h}},
 \qquad e_{n+1}\le\widehat e_{n+1}+\theta_n.
\end{equation}
For $\tau_nL_{\mathrm{Lip},h}\le\zeta<1$, iteration gives
\begin{equation}\label{eq:full-lipschitz-error}
 \max_{t_n\le T}e_n\le
 \exp\!\left(\frac{L_{\mathrm{Lip},h}T}{1-\zeta}\right)
 \left[e_0+\sum_{t_{j+1}\le T}
       (\|d_{j+1}\|_h+\|r_{j+1}\|_h+\theta_j)\right].
\end{equation}
This exhibits the two difficulties discussed in the introduction: both the error constant and the admissible step size depend on $L_{\mathrm{Lip},h}$. Even for a linear heat equation, the Lipschitz constant of the standard discrete Laplacian grows like $h^{-2}$, although the heat flow is contractive in $L^2$. The large constant in \eqref{eq:full-lipschitz} need not describe the actual growth of perturbations.

A natural question is then: why can PDE stability analysis avoid this dependence on a full Lipschitz constant? The answer is energy localization in PDE analysis.

\section{Why PDE stability does not require a full Lipschitz bound}\label{sec:pde-localization}
Energy localization is a standard approach in PDE analysis; see, for example, \cite{Kato1975}. It relies on two mechanisms: controlled energy propagation and one-sided stability within the energy-localized region.

Let $H$ be a Hilbert space and $V\subset H$ a space encoding the required regularity and moments. Consider the evolution equation
\begin{equation}\label{eq:pde-evolution}
 u'=\mathcal F(t,u).
\end{equation}
For sufficiently regular solutions, the two mechanisms take the following form.
\begin{enumerate}
\item \textbf{Controlled energy propagation.}
Choose a continuous energy $\E:V\to[0,\infty)$ with $\E(u)\ge\|u\|_H^2$ such that $\mathcal K_Q=\{u\in V:\E(u)\le Q^2\}$ bounds the required regularity and moments. While $u(t)\in\mathcal K_Q$, assume
\begin{equation}\label{eq:pde-energy-growth}
 \frac{\dd}{\dd t}\E(u(t))\le c_Q(1+\E(u(t))),\qquad c_Q\ge0.
\end{equation}
If $\E(u(0))<Q^2$, Gronwall's inequality prevents exit from $\mathcal K_Q$ on a sufficiently short interval determined by $Q$, $c_Q$, and the initial energy.

\item \textbf{One-sided stability in the localized region.}
For $U,V\in\mathcal K_Q$, assume
\begin{equation}\label{eq:pde-one-sided}
 \re\ip{U-V}{\mathcal F(t,U)-\mathcal F(t,V)}_H
 \le\mu_Q\|U-V\|_H^2.
\end{equation}
Here $\mu_Q\ge0$ depends on the energy region.
\end{enumerate}
For spatially discretized equations, $c_Q$ and $\mu_Q$ are required to be independent of the mesh.

For two solutions $u,v$ remaining in $\mathcal K_Q$, set $e(t)=\|u(t)-v(t)\|_H$. Subtracting their equations and applying \eqref{eq:pde-one-sided} gives
\[
 \frac12\frac{\dd}{\dd t}e(t)^2\le\mu_Q e(t)^2.
\]
Gronwall's inequality gives $e(t)\le e^{\mu_Qt}e(0)$ without a full Lipschitz bound.

To transfer this argument to SAT, the energy localization must survive both the step and the truncate operations. We next identify one-sided energies whose bounds are preserved by SVD truncation.

\section{One-sided energies bring localization analysis into SAT}\label{sec:framework}
Section~\ref{sec:pde-localization} shows how regularity bounds and one-sided stability work together in PDE analysis. To use the same argument for the SAT update in Section~\ref{sec:sat}, the energy region must also be preserved by truncation. We identify a class of one-sided energies with this property.

In the matrix space $H_h$ introduced in Section~\ref{sec:sat}, consider an energy formed from operators acting on one matrix factor at a time:
\begin{equation}\label{eq:energy}
 \E_h(Z)=\|Z\|_h^2+\sum_j\|C_{j,h}Z\|_h^2
              +\sum_k\|Z D_{k,h}\|_h^2
          =\ip{Z}{A_h Z}_h,
 \qquad \|Z\|_{\E,h}=\E_h(Z)^{1/2}.
\end{equation}
Here $A_h$ is the positive self-adjoint operator on the matrix space induced by this quadratic form. For $Q>0$, define the open energy ball $\Omega_Q=\{Z:\E_h(Z)<Q^2\}$ and its closure $\overline{\Omega_Q}=\{Z:\E_h(Z)\le Q^2\}$.

The gradient of $\E_h/2$ is $A_hZ=Z+\sum_j C_{j,h}^*C_{j,h}Z+\sum_k ZD_{k,h}D_{k,h}^*$. Each term is a left or right multiple of $Z$ and hence belongs to its fixed-rank tangent space. This gives tangent-space compatibility; the following lemma establishes compatibility with SAT truncation.

\begin{lemma}[SVD contraction]\label{lem:svd}
Let $\Tr_r Z$ be any partial SVD retaining a subset of the singular terms of $Z$. Then
\begin{equation}\label{eq:svdcontract}
 \|\Tr_r Z\|_h\le\|Z\|_h,\quad
 \|C_{j,h}\Tr_r Z\|_h\le\|C_{j,h}Z\|_h,\quad
 \|(\Tr_r Z)D_{k,h}\|_h\le\|Z D_{k,h}\|_h.
\end{equation}
In particular, $\E_h(\Tr_r Z)\le\E_h(Z)$, with no singular-value gap assumption.
\end{lemma}
\begin{proof}
Write $Z=\sum_\ell\sigma_\ell u_\ell v_\ell^*$. Orthogonality of the $v_\ell$ gives
$\|C Z\|_h^2=\sum_\ell\sigma_\ell^2\|C u_\ell\|^2$.
Orthogonality of the $u_\ell$ gives the analogous identity for $Z D$. Removing nonnegative summands proves each assertion.
\end{proof}

Lemma~\ref{lem:svd} identifies a class of energy-localized regions preserved by ordinary SVD truncation. It allows spatial regularity, velocity regularity, and moments to be controlled separately within the same energy. This is the structural property used below to keep the numerical trajectory in the region where the PDE stability estimate holds. A mixed term $\|CZD\|_h$ does not have the same contraction property in general.

\section{Existence, uniqueness, and robust error analysis}\label{sec:stability}
We now apply the localization argument of Section~\ref{sec:pde-localization} to the SAT scheme \eqref{eq:stage}--\eqref{eq:truncate}, using the one-sided energies from Section~\ref{sec:framework}.
For a family of spatial discretizations, the two PDE principles are complemented by a mesh-uniform initial energy bound. We call these the three principles for robust SAT error analysis.

\subsection{Three principles for robust SAT error analysis}\label{sec:three-principles}
\begin{enumerate}
\item \textbf{Mesh-uniform initial energy bound.}
For a fixed radius $Q>0$, the initial SAT states satisfy
\[
 \sup_h\E_h(Y^0)\le q_0<Q^2,
\]
where $q_0$ and $Q$ are independent of $h$. Thus the initial energy margin $Q^2-q_0$ is uniformly positive.

\item \textbf{Mesh-uniform controlled propagation of the one-sided energy.}
For every $Z\in\overline{\Omega_Q}$,
\[
 D\E_h(Z)[F_h(t,Z)]
 =2\re\ip{A_hZ}{F_h(t,Z)}_h
 \le c_Q(1+\E_h(Z)),
\]
with $c_Q\ge0$ independent of $h$. This is the pointwise form of \eqref{eq:pde-energy-growth} for the discrete energy.

\item \textbf{Mesh-uniform one-sided Lipschitz stability.}
For all $U,V\in\overline{\Omega_Q}$,
\begin{equation}\label{eq:onesided}
 \re\ip{U-V}{F_h(t,U)-F_h(t,V)}_h
 \le\mu_Q\|U-V\|_h^2,
\end{equation}
where $\mu_Q\ge0$ is independent of $h$.
\end{enumerate}
Under the step-size and residual conditions below, controlled energy growth establishes existence of the next implicit solution in the energy ball, and one-sided Lipschitz stability makes it unique there for each prescribed residual. SVD contraction preserves the localization after truncation. The same stability estimate then controls propagation of one-step errors.

\subsection{Existence and uniqueness under the three principles}\label{sec:implicit-stages}
For implicit SAT, Rodgers and Venturi \cite[Section~3]{RodgersVenturi2023} invoke the implicit function theorem to establish existence of the implicit solution. Here we establish existence in the energy ball by a topological argument: an energy estimate excludes boundary zeros of a homotopy, so the nonzero Brouwer degree ensures an interior root. This argument requires continuity and controlled energy growth, without a Jacobian invertibility assumption.

Theorem~\ref{thm:physical} ensures that, under its stated conditions, if $Y^n\in\Omega_Q$ then a next SAT state $Y^{n+1}\in\Omega_Q$ exists.

We bound the full residual in \eqref{eq:stage} by
\begin{equation}\label{eq:residualbounds}
 \|r_{n+1}\|_h\le\eps_n,\qquad
 \|r_{n+1}\|_{\E,h}\le\eps_{E,n}.
\end{equation}
The energy-residual bound enters the a priori existence and energy-propagation argument. The discrete $L^2$ residual bound enters the error estimate once the stages are localized.

\begin{theorem}[Existence, uniqueness, and energy propagation]\label{thm:physical}
Suppose $F_h$ is continuous, the energy-growth and one-sided Lipschitz conditions above hold on $\overline{\Omega_Q}$, and $\E_h(Y^n)\le q_n<Q^2$. Choose $\eta>0$ and set
\begin{equation}\label{eq:budget}
 q_{n+1}=\frac{q_n+c_Q\tau_n+\eps_{E,n}^2/(\eta\tau_n)}
                       {1-(c_Q+\eta)\tau_n}.
\end{equation}
If the denominator is positive, $q_{n+1}<Q^2$, and $\tau_n\mu_Q<1$, then for every fixed $r$ with $\|r\|_{\E,h}\le\eps_{E,n}$ the equation
$Z-Y^n-\tau_nF_h(t_{n+1},Z)=r$
has a unique solution in $\overline{\Omega_Q}$. This solution lies in $\Omega_Q$, satisfies $\E_h(Z)\le q_{n+1}$, and its SVD truncation satisfies the same bound.
\end{theorem}
\begin{proof}
We first prove existence by Brouwer degree on the bounded open energy ball $\Omega_Q$, viewing $H_h$ as a real finite-dimensional space. For the definition, normalization, homotopy invariance, and existence property of this degree, see \cite[Chapter~1]{Deimling1985}. Fix an admissible residual $r$ and define
\[
 G_\lambda(Z)=Z-Y^n-\lambda\tau_nF_h(t_{n+1},Z)-\lambda r,
 \qquad 0\le\lambda\le1.
\]
At $\lambda=0$, $G_0(Z)=Z-Y^n$ has its unique zero $Y^n$ in $\Omega_Q$. Normalization and translation give $\deg(G_0,\Omega_Q,0)=1$.

To preserve this degree along the homotopy, we exclude zeros on $\partial\Omega_Q$. For any $Z$ on this boundary, $\E_h(Z)=Q^2$. The energy-growth condition and Young's inequality give
\[
\begin{aligned}
 2\re\ip{A_hZ}{G_\lambda(Z)}_h
 &\ge Q^2-q_n-\lambda\tau_nc_Q(1+Q^2)
       -\eta\tau_nQ^2-\frac{\eps_{E,n}^2}{\eta\tau_n}\\
 &\ge [1-(c_Q+\eta)\tau_n](Q^2-q_{n+1})>0.
\end{aligned}
\]
Thus $G_\lambda(Z)\ne0$ on $\partial\Omega_Q$ for every $\lambda\in[0,1]$. Homotopy invariance yields $\deg(G_1,\Omega_Q,0)=1$. Since a nonzero degree implies a zero in the domain, the implicit equation has a solution in $\Omega_Q$.

We now show that $\E_h(Z)\le q_{n+1}$ for any such solution. Taking the energy inner product of its equation with $Z$ gives
\begin{equation}\label{eq:energyidentity}
 \E_h(Z)-\E_h(Y^n)+\E_h(Z-Y^n)
 =2\tau_n\re\ip{A_h Z}{F_h(t_{n+1},Z)}_h
   +2\re\ip{A_h Z}{r}_h.
\end{equation}
Using the energy-growth condition and
$2\|Z\|_{\E,h}\|r\|_{\E,h}\le\eta\tau_n\E_h(Z)+\eps_{E,n}^2/(\eta\tau_n)$
and discarding the nonnegative difference term gives $\E_h(Z)\le q_{n+1}$. Lemma~\ref{lem:svd} transfers this bound to its SVD truncation.

For uniqueness, let $Z_1,Z_2\in\overline{\Omega_Q}$ solve the stage equation for the same $Y^n$ and $r$. Subtracting their equations and pairing with $W=Z_1-Z_2$ gives the following inequality, which forces $W=0$ because $\tau_n\mu_Q<1$:
\[
 \|W\|_h^2
 =\tau_n\re\ip{W}{F_h(t_{n+1},Z_1)-F_h(t_{n+1},Z_2)}_h
 \le\tau_n\mu_Q\|W\|_h^2.
\]
\end{proof}

To obtain a common time interval for the family of discretizations, suppose $\E_h(Y^0)\le q_0<Q^2$ and $\eps_{E,n}\le C_E\tau_n$, with $C_E$ independent of $h$. Fix $\eta>0$ and $0<\zeta<1$, and set
\begin{equation}\label{eq:uniformsteps}
 a_Q=c_Q+\eta,\qquad b_Q=c_Q+C_E^2/\eta,\qquad
 \tau_0=\frac{\zeta}{\max\{a_Q,\mu_Q\}}.
\end{equation}
Choose $T>0$ such that
\begin{equation}\label{eq:uniformtime}
 B(T):=e^{a_QT/(1-\zeta)}
       \left(q_0+\frac{b_Q}{a_Q}\right)-\frac{b_Q}{a_Q}<Q^2.
\end{equation}
For $\tau_n\le\tau_0$, the recursion \eqref{eq:budget} gives
$q_{n+1}+b_Q/a_Q\le(q_n+b_Q/a_Q)/(1-a_Q\tau_n)$.
Iteration using $-\log(1-z)\le z/(1-\zeta)$ for $0\le z\le\zeta$ yields
$q_n\le B(t_n)\le B(T)<Q^2$ for $t_n\le T$.
Theorem~\ref{thm:physical} therefore applies successively: each prescribed admissible residual has a unique implicit root in the energy ball, and SVD truncation preserves the bound.
Both $T$ and $\tau_0$ are independent of the mesh. The step restriction depends on energy growth and one-sided stability, avoiding the full Lipschitz restriction in \cite[Section~4.1]{RodgersVenturi2023} and \cite[Section~3.3.2]{AppeloCheng2025}.

\subsection{Robust error analysis of implicit SAT under the three principles}\label{sec:robust-error}
Use the comparison states and defects from Section~\ref{sec:sat}, and assume $X^n\in\overline{\Omega_Q}$. Pairing \eqref{eq:sat-comparison-identity} with $\widehat Y^{n+1}-X^{n+1}$ and applying \eqref{eq:onesided} yields
\begin{equation}\label{eq:impliciterror}
 \widehat e_{n+1}\le e_n+\|d_{n+1}\|_h+\eps_n
                 +\tau_n\mu_Q\widehat e_{n+1},
 \qquad e_{n+1}\le\widehat e_{n+1}+\theta_n,
\end{equation}
Here $e_n$ and $\widehat e_{n+1}$ are the errors defined in Section~\ref{sec:sat}. This replaces the full Lipschitz estimate \eqref{eq:full-lipschitz-recursion} by the localized one-sided estimate.
Only $\|r_{n+1}\|_h\le\eps_n$ is needed here: an energy-residual bound is unnecessary when stage membership in $\overline{\Omega_Q}$ is established directly.

\begin{theorem}[Robust error bound]\label{thm:linearstability}
Suppose all stages and comparison states lie in $\overline{\Omega_Q}$, $\|r_{n+1}\|_h\le\eps_n$, \eqref{eq:onesided} holds, and $\mu_Q\tau_n\le\zeta<1$. Then
\begin{equation}\label{eq:linearerror}
 \max_{t_n\le T}e_n\le
 \exp\!\left(\frac{\mu_QT}{1-\zeta}\right)
 \left[e_0+\sum_{t_{j+1}\le T}
       (\|d_{j+1}\|_h+\eps_j+\theta_j)\right].
\end{equation}
\end{theorem}
\begin{proof}
Move $\mu_Q\tau_n\widehat e_{n+1}$ to the left in \eqref{eq:impliciterror}. Iterate the resulting inequality and use
$-\log(1-z)\le z/(1-\zeta)$ for $0\le z\le\zeta$.
\end{proof}

\subsection{Necessity of energy constraints: residual control alone may select an incorrect solution}\label{sec:root-selection}
We define an admissible implicit stage to be a root of the stage equation in the prescribed energy-norm ball $\overline{\Omega_Q}$; roots outside this ball are excluded from the SAT update. Theorem~\ref{thm:physical} guarantees existence and uniqueness for each prescribed residual under its energy and step-size conditions. Theorem~\ref{thm:linearstability} gives robust error bounds for the resulting SAT iterates.

A scalar example shows why the residual alone cannot select the correct root. Consider $y'=y^3$ with $y_n=0$, whose exact solution remains zero. The backward-Euler equation has three roots:
\begin{equation}\label{eq:cubic-root-example}
 z-\tau z^3=0,\qquad z=0,\quad z=\pm\tau^{-1/2}.
\end{equation}
All three have zero residual, but the errors of the two nonzero roots grow as $\tau\to0$. For a fixed energy ball $|z|\le Q$, the condition $\tau<Q^{-2}$ excludes both exterior roots and retains only $z=0$.

Residual-based stopping criteria are used in the inexact Newton solver for implicit SAT \cite[Section~4.1]{RodgersVenturi2023}. Such criteria alone cannot distinguish the correct solution among several algebraic roots; convergence to an inconsistent solution branch is also documented for implicit ODE methods \cite{GreenPatrickSpiteri2019}. Our energy constraint supplies the additional selection criterion, with existence, uniqueness, and robust error guarantees under the conditions stated above. A solver must select an admissible stage in $\overline{\Omega_Q}$, for example by continuation from $Y^n$ and an energy acceptance test. Direct acceptance checks $\E_h(Z)\le Q^2$ and $\|r\|_h\le\eps_n$, with $\tau_n\mu_Q\le\zeta<1$ enforced by the step controller.
The two-dimensional reaction--diffusion experiment in Section~\ref{sec:numerical-root-selection} demonstrates the practical role of this constraint: residual-based Newton backtracking, initialized at the preceding SAT state, converges to exterior roots, whereas adding the energy constraint to the candidate updates recovers the correct decaying trajectory.

\section{Semilinear parabolic equations}\label{sec:parabolic}
We consider scalar reaction--diffusion equations
\begin{equation}\label{eq:generalreaction}
 u_t=\kappa\Delta u+f(u)\quad\hbox{in }(0,1)^2,\qquad u(0)=u_0,
\end{equation}
with fixed $\kappa>0$ and homogeneous Dirichlet boundary conditions $u|_{\partial(0,1)^2}=0$.
Assume $f\in C^2(\R)$; only bounds on $f'$ and $f''$ on bounded intervals will enter the estimates.
If $f(0)\ne0$, the substitution $u=v+\phi$, where $-\kappa\Delta\phi=f(0)$ and $\phi|_{\partial(0,1)^2}=0$, gives $v_t=\kappa\Delta v+\widetilde f(v+\phi)$ with $\widetilde f(s)=f(s)-f(0)$ and $\widetilde f(0)=0$; hereafter we assume $f(0)=0$.

We use centered second differences at the interior nodes $(ih,jh)$, $1\le i,j\le m$, with $h=(m+1)^{-1}$.
With $B_{x,h}=B_{y,h}=-\kappa\Delta_h^{\rm Dir}$ for the one-dimensional Dirichlet Laplacian, the semidiscrete equation is
\begin{equation}\label{eq:parabolic}
 X_h'=F_h(X_h),\qquad
 F_h(Z)=-L_hZ+f^\circ(Z),\qquad
 L_hZ=B_{x,h}Z+ZB_{y,h}.
\end{equation}
Here $f^\circ$ denotes entrywise application, $X_{h,0}$ samples the continuous representative of $u_0$, and $\|Z\|_h=h\|Z\|_F$.
Both difference matrices are positive definite and self-adjoint, and their left and right actions commute.

Let $\widetilde u_0$ be the odd extension of $u_0$ in both coordinates, periodically continued with period two.
Writing $a_k$ for the Fourier coefficients of this extension on its period cell, define
\begin{equation}\label{eq:initial-fourier}
 \mathcal A_s(u_0)^2
 =\sum_{k\in\mathbb Z^2}(1+|k|^2)^s|a_k|^2.
\end{equation}
To verify the three principles in Section~\ref{sec:three-principles}, we assume $\mathcal A_2(u_0)<\infty$.
For the first-order error estimate below, we assume in addition $\mathcal A_{2+\sigma}(u_0)<\infty$ for some fixed $0<\sigma<1/2$.

\subsection{Verification of the three principles}\label{sec:parabolic-principles}
Define the one-sided graph energy
\begin{equation}\label{eq:graphenergy}
 \E_h(Z)=\|Z\|_h^2+\|B_{x,h}Z\|_h^2+\|ZB_{y,h}\|_h^2,
 \qquad \|Z\|_{\E,h}=\E_h(Z)^{1/2}.
\end{equation}
We verify the three principles for this energy.

\begin{enumerate}
\item \textbf{Mesh-uniform initial energy bound.}
Extend $Z$ oddly across both boundaries, including zero boundary values, to a periodic array $z$.
Set $n=2(m+1)$ and $K_n=\{-n/2,\ldots,n/2-1\}^2$.
Normalize the discrete Fourier transform by
$\widehat z_k=n^{-2}\sum_{j\in\{0,\ldots,n-1\}^2}z_j e^{-2\pi\mathrm i k\cdot j/n}$.
Parseval's identity gives $\sum_k|\widehat z_k|^2=\|Z\|_h^2$.
For a scalar Fourier frequency $p\in\{-n/2,\ldots,n/2-1\}$, the one-dimensional difference symbol is
\[
 \lambda_h(p)=4\kappa h^{-2}\sin^2(\pi p/n).
\]
Its values at $p=1,\ldots,m$ are the eigenvalues of each one-dimensional Dirichlet matrix $B_{x,h}$ and $B_{y,h}$.
For a two-dimensional frequency $k=(k_1,k_2)\in K_n$, their left and right actions have symbols $\lambda_h(k_1)$ and $\lambda_h(k_2)$, respectively.
Since $nh=2$ and $|p|\le n/2$, we have
$4\kappa p^2\le\lambda_h(p)\le\pi^2\kappa p^2$, uniformly in $h$.
Consequently,
\begin{equation}\label{eq:parabolic-energy-fourier}
 \begin{aligned}
 \E_h(Z)&=\sum_{k\in K_n}
 \bigl(1+\lambda_h(k_1)^2+\lambda_h(k_2)^2\bigr)|\widehat z_k|^2\\
 &\asymp_\kappa \sum_{k\in K_n}(1+|k|^2)^2|\widehat z_k|^2.
 \end{aligned}
\end{equation}

Write $w(k)=1+|k|^2$.
Since $\mathcal A_2(u_0)<\infty$, the Fourier coefficients are absolutely summable and sampling gives the aliasing identity
\[
 \widehat x_{h,0}(k)=\sum_{\ell\in\mathbb Z^2}a_{k+n\ell}.
\]
For every $s>1$, we have
\[
 \sup_{n,k\in K_n}
 \sum_{\ell\in\mathbb Z^2}
 \left(\frac{w(k)}{w(k+n\ell)}\right)^s\le C_s,
\]
since the finitely many nearest aliases have bounded ratios, and the remaining terms are bounded by $C_s(1+|\ell|^2)^{-s}$.
Weighted Cauchy--Schwarz therefore gives
\[
 w(k)^s|\widehat x_{h,0}(k)|^2
 \le C_s\sum_{\ell\in\mathbb Z^2}
 w(k+n\ell)^s|a_{k+n\ell}|^2.
\]
Summing over the disjoint aliasing classes yields
\begin{equation}\label{eq:initial-sampling}
 \sum_{k\in K_n}w(k)^s|\widehat x_{h,0}(k)|^2
 \le C_s\mathcal A_s(u_0)^2,\qquad s\ge2.
\end{equation}
Taking $s=2$ and using \eqref{eq:parabolic-energy-fourier} gives
\begin{equation}\label{eq:cubicinitialenergy}\E_h(X_{h,0})\le C_\kappa\mathcal A_2(u_0)^2=:q_*.
\end{equation}
Taking $Y^0=X_{h,0}$, or any SVD truncation of it, gives $\E_h(Y^0)\le q_*$.
Choose $Q^2>q_*$ to obtain a uniform positive initial energy margin.

\item \textbf{Mesh-uniform controlled propagation of the one-sided energy.}
The inverse Fourier transform and Cauchy--Schwarz give
\[
 \begin{aligned}
 \|Z\|_{\ell^\infty}
 &\le\sum_{k\in K_n}|\widehat z_k|\\
 &\le\left(\sum_{k\in\mathbb Z^2}(1+|k|^2)^{-2}\right)^{1/2}
       \left(\sum_{k\in K_n}(1+|k|^2)^2|\widehat z_k|^2\right)^{1/2}\\
 &\le C_S\|Z\|_{\E,h},
 \end{aligned}
\]
where the infinite sum is finite and \eqref{eq:parabolic-energy-fourier} makes $C_S$ independent of $h$.
Fix $Z\in\overline{\Omega_Q}$ and set
\[
 \begin{gathered}
 M_Q=C_SQ,\qquad
 M_{j,Q}=\sup_{|s|\le M_Q}|f^{(j)}(s)|\quad(j=1,2),\\
 K_Q=M_{1,Q}+3M_QM_{2,Q}.
 \end{gathered}
\]

Let $a=\delta_x^+Z$, so $\delta_x^-a=\delta_{xx}Z$, and write
$\|a\|_{4,h}^4=h^2\sum|a_{ij}|^4$.
Discrete summation by parts and the difference of cubes give
\[
 \begin{aligned}
 \|a\|_{4,h}^4
 &=-h^2\sum_{i,j}Z_{ij}(\delta_{xx}Z)_{ij}
       (a_{ij}^2+a_{ij}a_{i-1,j}+a_{i-1,j}^2)\\
 &\le3\|Z\|_{\ell^\infty}\|\delta_{xx}Z\|_h\|a\|_{4,h}^2.
 \end{aligned}
\]
The norm of $a$ sums over all edges, including those adjacent to the zero boundary, and the sum on the right is over interior nodes.
Each of the three quadratic terms has $\ell_h^2$ norm at most $\|a\|_{4,h}^2$ by H\"older's inequality.
Dividing, unless $a=0$, gives
\begin{equation}\label{eq:discrete-gradient-four}
 \|\delta_x^+Z\|_{4,h}^2
 \le3\|Z\|_{\ell^\infty}\|\delta_{xx}Z\|_h.
\end{equation}
Taylor's formula, including the boundary nodes since $f(0)=0$, gives pointwise
\[
 \begin{aligned}
 |\delta_{xx}f^\circ(Z)|
 &\le M_{1,Q}|\delta_{xx}Z|\\
 &\quad+\tfrac12M_{2,Q}\bigl(|\delta_x^+Z|^2+|\delta_x^-Z|^2\bigr).
 \end{aligned}
\]
Combining this with \eqref{eq:discrete-gradient-four} yields
$\|B_{x,h}f^\circ(Z)\|_h\le K_Q\|B_{x,h}Z\|_h$.
The same argument in the other direction gives
$\|f^\circ(Z)B_{y,h}\|_h\le K_Q\|ZB_{y,h}\|_h$.
Finally, $f(0)=0$ implies $\|f^\circ(Z)\|_h\le M_{1,Q}\|Z\|_h$.
These estimates give
\begin{equation}\label{eq:cubicsobolev}
 \|f^\circ(Z)\|_{\E,h}\le K_Q\|Z\|_{\E,h}
 \quad(\E_h(Z)\le Q^2),
\end{equation}
with constants independent of $h$.

Write $\E_h(Z)=\ip{Z}{A_hZ}_h$, where
$A_hZ=Z+B_{x,h}^2Z+ZB_{y,h}^2$.
Since $A_h$ and $L_h$ commute and are nonnegative, the diffusion contribution is nonpositive.
Thus, throughout the closed energy ball,
\begin{equation}\label{eq:cubicenergy}
 \begin{aligned}
 D\E_h(Z)[F_h(Z)]
 &=-2\ip{A_hZ}{L_hZ}_h+2\ip{Z}{f^\circ(Z)}_{\E,h}\\
 &\le2K_Q\E_h(Z)
 \le c_Q(1+\E_h(Z)),\qquad c_Q=2K_Q.
 \end{aligned}
\end{equation}

\item \textbf{Mesh-uniform one-sided stability.}
For $U,V\in\overline{\Omega_Q}$ every entry lies in $[-M_Q,M_Q]$ by the maximum-norm estimate above.
Set
\[
 \mu_Q=\max\left\{0,\sup_{|s|\le M_Q}f'(s)\right\}.
\]
The scalar mean-value theorem and diffusion dissipation yield, with $W=U-V$,
\begin{equation}\label{eq:paraboliccancel}
 \begin{aligned}
 \ip{W}{F_h(U)-F_h(V)}_h
 &=-\ip{W}{L_hW}_h
   +h^2\sum_{i,j}W_{ij}\bigl(f(U_{ij})-f(V_{ij})\bigr)\\
 &\le\mu_Q\|W\|_h^2.
 \end{aligned}
\end{equation}
No bound on $\|L_h\|=O(h^{-2})$ enters $c_Q$ or $\mu_Q$.
\end{enumerate}

\subsection{Mesh-robust first-order error}
Let $X_h(t)$ be the exact solution of \eqref{eq:parabolic} with initial value $X_{h,0}$.
The additional initial regularity in \eqref{eq:initial-fourier} supplies a mesh-uniform time-consistency budget.

\begin{lemma}[Mesh-uniform time consistency]\label{lem:cubic-time-consistency}
Assume $\mathcal A_{2+\sigma}(u_0)<\infty$ for some $0<\sigma<1/2$.
Choose $Q^2>q_*$ and a common interval $[0,T]$ on which $X_h(t)\in\overline{\Omega_Q}$.
For any time grid with maximum step $\tau=\max_n\tau_n$, the comparison states $X^n=X_h(t_n)$ satisfy
\begin{equation}\label{eq:cubicdefect}
 \int_0^T\|X_h''(t)\|_h\dd t\le C_{Q,T},\qquad
 \sum_{t_{n+1}\le T}\|d_{n+1}\|_h\le C_{Q,T}\tau,
\end{equation}
with constants independent of $h$ and the time grid.
\end{lemma}
\begin{proof}
The integral bound is proved in Appendix~\ref{app:cubic-time-consistency}.
For the defects defined in \eqref{eq:defect},
\[
 \begin{aligned}
 d_{n+1}
 &=\int_{t_n}^{t_{n+1}}\bigl(X_h'(s)-X_h'(t_{n+1})\bigr)\dd s\\
 &=-\int_{t_n}^{t_{n+1}}(s-t_n)X_h''(s)\dd s.
 \end{aligned}
\]
Summing gives the second estimate in \eqref{eq:cubicdefect}.
\end{proof}

\begin{theorem}[SAT error for semilinear reaction--diffusion equations]\label{thm:parabolic}
Consider \eqref{eq:generalreaction}--\eqref{eq:parabolic} with $f\in C^2(\R)$, $f(0)=0$, and $\mathcal A_{2+\sigma}(u_0)<\infty$ for some $0<\sigma<1/2$.
Take $Y^0=X_{h,0}$ or an SVD truncation of it, choose $Q^2>q_*$ as in \eqref{eq:cubicinitialenergy}, and assume $\eps_{E,n}\le C_E\tau_n$.
Use the mesh-independent interval and step-size conditions \eqref{eq:uniformsteps}--\eqref{eq:uniformtime}, with $c_Q,\mu_Q$ from \eqref{eq:cubicenergy}--\eqref{eq:paraboliccancel} and initial budget $q_0=q_*$.
Then the energy-admissible SAT iterates satisfy
\begin{equation}\label{eq:parabolicerror}
 \max_{t_n\le T}\|Y^n-X_h(t_n)\|_h
 \le C_{Q,T}\left[
 \|Y^0-X_{h,0}\|_h+\tau+\sum_{t_{n+1}\le T}(\eps_n+\theta_n)\right].
\end{equation}
The constant may depend on $\kappa,f,\sigma$ and the initial regularity bound, but is independent of the mesh and numerical singular values.
\end{theorem}
\begin{proof}
The preceding energy estimates put the stages and comparison states in the same ball.
Apply Theorem~\ref{thm:linearstability} with \eqref{eq:paraboliccancel} and the defect budget \eqref{eq:cubicdefect}.
\end{proof}

\Needspace{4\baselineskip}
\begin{remark}[Energy residuals and direct energy acceptance]\label{rem:energy-residual-acceptance}
The energy-residual condition $\eps_{E,n}\le C_E\tau_n$ guarantees a priori localization of the implicit stages. For residuals with substantial high-frequency content, enforcing it can require a tighter discrete $L^2$ solver tolerance as the mesh is refined. Accordingly, Algorithm~\ref{alg:energy-sat}, developed below, checks stage membership in the energy ball directly instead of imposing this energy-residual condition. If these acceptance checks give $\E_h(\widehat Y^{n+1})\le Q^2$, the exact semidiscrete solution remains in the same ball, and $\tau_n\mu_Q\le\zeta<1$, then SVD contraction, Theorem~\ref{thm:linearstability}, and Lemma~\ref{lem:cubic-time-consistency} give the same error estimate using the discrete $L^2$ residual and SVD-tail bounds. This conclusion applies on the accepted trajectory and requires no energy-residual bound. We emphasize that either form of localization, the energy-residual condition of this theorem or the direct energy checks used below, yields a mesh-robust error estimate: in both cases, the constant in \eqref{eq:parabolicerror} is independent of the mesh.
\end{remark}

As a concrete instance, the problem used in \cite[Section~2]{OstermannPiazzolaWalach2019} is
\begin{equation}\label{eq:cubicmodel}
 \begin{aligned}
 &u_t=\kappa\Delta u+u^3\quad\hbox{in }(0,1)^2,\qquad
 u|_{\partial(0,1)^2}=0,\\
 &u_0(x,y)=16x(1-x)y(1-y),\qquad \kappa=1/50.
 \end{aligned}
\end{equation}
Its sine coefficients satisfy $|a_{k\ell}|\le Ck^{-3}\ell^{-3}$, so
$\mathcal A_{2+\sigma}(u_0)<\infty$ whenever $0<\sigma<1/2$:
by symmetry of the two coordinate contributions, the weighted sum is bounded by a constant times
$(\sum_{k\ge1}k^{-2+2\sigma})(\sum_{\ell\ge1}\ell^{-6})$.
For $p_h=(ih(1-ih))_{i=1}^m$, the elementary identity
$-\Delta_h^{\rm Dir}p_h=2\boldsymbol1$, where $\boldsymbol1$ is the all-ones vector, gives the sharper initial bound
$\E_h(X_{h,0})\le1+128\kappa^2$.

Thus first-order time convergence follows when the accumulated solver and truncation errors are $O(\tau)$.
For the cubic example, the principal estimate for low-rank Lie--Trotter splitting in \cite[Theorem~1]{OstermannPiazzolaWalach2019} is
\begin{equation}\label{eq:ostermanncomparison}
 \|X_h(t_n)-Y_{\rm LT}^n\|_h
 \le C_0\tau(1+|\log\tau|)+C_1\delta_0+C_2\eps_{\rm tan}.
\end{equation}
Here $\delta_0$ is the initial error and $\eps_{\rm tan}$ bounds the normal component of the nonlinear vector field.
Our full implicit update requires no stiff/non-stiff time splitting and gives the $O(\tau)$ time-error bound \eqref{eq:parabolicerror}, with measured residuals and SVD tails.

\section{The Vlasov--Poisson system}\label{sec:vp}
We next treat the original Vlasov--Poisson equation on the whole velocity line.

\subsection{Choice of the spatial discretization}\label{sec:vp-discretization}
Let $x\in\T=\R/(2\pi\mathbb Z)$ and $v\in\R$. We use the sign convention
\begin{equation}\label{eq:vp-continuous}
 \begin{aligned}
 &f_t+v f_x-E[f]f_v=0,\qquad
 E[f]_x=\chi(\rho_f-\overline{\rho_f}),\\
 &\int_{\T}E[f]\dd x=0,\qquad
 \rho_f=\int_{\R}f\dd v,
 \end{aligned}
\end{equation}
where the bar denotes the spatial mean. With this convention, $\chi=-1$ is repulsive plasma coupling and $\chi=1$ is attractive coupling. The spatial boundary condition is periodic, and velocity remains on the whole real line. Define
\begin{equation}\label{eq:vp-h-transform}
 w(v)=\sqrt{1+v^2},\quad g=wf,\quad
 b(v)=\frac{w'(v)}{w(v)}=\frac{v}{1+v^2},\quad T=\partial_v-b(v).
\end{equation}
Then \eqref{eq:vp-continuous} is exactly equivalent to
\begin{equation}\label{eq:vp-h-g}
 g_t=\LL_g g:=-v\partial_xg+E[g]Tg,\qquad
 \rho[g]=\int_{\R}\frac{g}{w}\dd v.
\end{equation}
Here and below $E[g]$ is computed from $\rho[g]$, not from $\int g\,dv$.
All norms without subscripts are in $L^2(\T\times\R)$.
The corresponding norm for $f$ is $\norm{wf}$, and
\begin{equation}\label{eq:vp-h-density}
 \norm{\rho[z]}_{L_x^2}\le\sqrt\pi\norm z,\qquad
 \norm{E[z]}_{L_x^\infty}\le c_E\norm z,\qquad c_E=\frac\pi{\sqrt6}.
\end{equation}
Indeed, $\int_{\R}w^{-2}\,dv=\pi$. With the orthonormal spatial modes
$e_k(x)=e^{ikx}/\sqrt{2\pi}$, Poisson inversion and
$\sum_{k\ne0}k^{-2}=\pi^2/3$ give the second bound. Also
$\norm b_\infty=1/2$, and every fixed-order derivative of $b$ is bounded.
For an integer $m\ge0$, the same argument, and the one-dimensional Sobolev
embedding for derivatives of $E$, give
\begin{equation}\label{eq:vp-h-fieldhigher}
 \norm{E[z]}_{W^{m,\infty}(\T)}\le C_m\norm z_{H_x^mL_v^2}.
\end{equation}
\subsubsection{Fourier--Hermite Galerkin projection}
Recall the orthonormal Fourier modes $e_k(x)=e^{ikx}/\sqrt{2\pi}$, $k\in\mathbb Z$.
Let $\phi_n$ be the normalized real Hermite functions and set
\[
 \begin{aligned}
 A&=1-\partial_x^2,& A e_k&=(1+k^2)e_k,\\
 \HH&=1-\partial_v^2+v^2,& \HH\phi_n&=(2n+2)\phi_n.
 \end{aligned}
\]
For $K\ge1$ and $N\ge1$, let $\PP=\PP_{K,N}$ denote orthogonal projection onto
\[
 S_{K,N}=\operatorname{span}\{e_k\phi_n:|k|\le K,\ 0\le n<N\}.
\]
Work in the real subspace of this space, equivalently with Fourier coefficients
satisfying $U_{-k,n}=\overline{U_{k,n}}$. The semidiscrete equation is
\begin{equation}\label{eq:vp-hermite-ode}
 u'=F_{K,N}(u)=\PP\LL_u u,\qquad u(0)=\PP g_0.
\end{equation}
If $u=\sum U_{kn}e_k\phi_n$, then $\norm u=\norm U_F$ and
\begin{equation}\label{eq:vp-h-matrix}
 U'=-D_xUV_N^\top+C_{E[U]}U T_N^\top,\qquad
 T_N=D_N-B_N.
\end{equation}
The matrices are
\[
 \begin{aligned}
 (D_x)_{kk}&=ik,& (V_N)_{mn}&=\ip{\phi_m}{v\phi_n},\\
 (D_N)_{mn}&=\ip{\phi_m}{\partial_v\phi_n},&
 (B_N)_{mn}&=\ip{\phi_m}{b\phi_n}.
 \end{aligned}
\]
The field uses the density vector $d_n=\int\phi_n/w\,dv$ and
$\rho_k=\sum_n U_{kn}d_n$; multiplication by $E$ is exact Fourier convolution.
Here $V_N,B_N$ are real symmetric and $D_N$ is real skew-symmetric.
This retains the usual left/right matrix actions needed for a factor solver.

\subsection{Verification of the three principles}\label{sec:vp-principles}
We verify the principles of Section~\ref{sec:three-principles} for the two graph energies defined below.

\subsubsection{One-sided graph energies}
For $s=1,2$, define
\begin{equation}\label{eq:vp-hermite-energy}
 \begin{aligned}
 \E_s(u)&=\norm u^2+\norm{A^s u}^2+\norm{\HH^s u}^2\\
 &=\norm U_F^2+\norm{A_K^sU}_F^2+\norm{UH_N^s}_F^2,
 \end{aligned}
\end{equation}
where $A_K=\operatorname{diag}(1+k^2)$ and
$H_N=\operatorname{diag}(2n+2)$. Thus $s=1$ controls second-order derivatives
and moments; $s=2$ controls fourth-order ones. The oscillator controls both
velocity derivatives and velocity moments.
Crucially,
\begin{equation}\label{eq:vp-h-commute}
 \PP A^s=A^s\PP,\qquad \PP\HH^s=\HH^s\PP.
\end{equation}
Hence $\E_s(\PP z)\le\E_s(z)$. Ordinary partial SVD of $U$ also gives
\begin{equation}\label{eq:vp-h-svd}
 \E_s(\Tr_r U)\le\E_s(U).
\end{equation}
This follows from Lemma~\ref{lem:svd} and requires no singular-value gap.
SVD here is in the coefficient norm for $g$, equivalently the weighted norm
for $f$. Dividing the retained velocity factors by $w$ preserves the separated
rank.

We record the mixed-regularity estimate used in the commutator bounds.
\begin{lemma}\label{lem:vp-h-mixed}
Let $s=1$ or $2$. For nonnegative integers $a,p,q$ with $a+p+q\le2s$,
\begin{equation}\label{eq:vp-h-mixed}
 \norm{\partial_x^a v^p\partial_v^q u}\le C_s\E_s(u)^{1/2}.
\end{equation}
The same bound holds for any ordering of the $p+q$ velocity multiplication
and differentiation operators. A bounded coefficient multiplying the resulting
expression contributes its supremum norm. These constants are independent of $K,N$.
\end{lemma}
\begin{proof}
The Hermite identities
\[
 \begin{aligned}
 v\phi_n&=\sqrt{n/2}\,\phi_{n-1}+\sqrt{(n+1)/2}\,\phi_{n+1},\\
 \partial_v\phi_n&=\sqrt{n/2}\,\phi_{n-1}-\sqrt{(n+1)/2}\,\phi_{n+1}
 \end{aligned}
\]
show that a word of length $r$ in $v,\partial_v$ is bounded from the graph
norm of $\HH^{r/2}$ to $L^2$: there are at most $2^r$ shifts, with coefficients
bounded by $C_r(n+1)^{r/2}$. This is the elementary $L^2$ case of
Hermite--Sobolev estimates \cite{BongioanniTorrea2006}.
Applying this to each spatial Fourier mode, and using
\[
 |k|^{2a}(2n+2)^{p+q}
 \le C_s\bigl(1+(1+k^2)^{2s}+(2n+2)^{2s}\bigr)
\]
for $a+p+q\le2s$, proves the result.
\end{proof}

\subsubsection{First principle: mesh-uniform initial energy}
If $g_0=wf_0$ satisfies $\E_s(g_0)<\infty$, projection and \eqref{eq:vp-h-svd} yield
\begin{equation}\label{eq:vp-h-initial}
 \sup_{K,N}\E_s\bigl(\Tr_r\PP_{K,N}g_0\bigr)
 \le\E_s(g_0)=q_0.
\end{equation}
Thus any $Q^2>q_0$ supplies a common initial margin. Smooth Maxwellian data
with smooth spatial modulation satisfy this condition for both $s=1$ and $2$.

\subsubsection{Second principle: controlled energy growth}
\begin{theorem}\label{thm:vp-h-growth}
For $s=1,2$, every real $u\in S_{K,N}$ satisfies
\begin{equation}\label{eq:vp-h-growth}
 D\E_s(u)[F_{K,N}(u)]
 \le C_s\bigl(1+\E_s(u)^{1/2}\bigr)\E_s(u).
\end{equation}
Consequently, on $\E_s(u)\le Q^2$, the second principle holds with
$c_{s,Q}=C_s(1+Q)$, independent of $K,N$.
\end{theorem}
\begin{proof}
Fix $u$ and write $E=E[u]$. For any sufficiently regular $z$,
integration by parts gives
\begin{equation}\label{eq:vp-h-frozen}
 \begin{aligned}
 \re\ip z{\LL_u z}=-\int_{\T\times\R}E(x)b(v)|z(x,v)|^2\,dx\,dv,
 \\
 |\re\ip z{\LL_u z}|\le\tfrac12\norm E_\infty\norm z^2.
 \end{aligned}
\end{equation}
The streaming and the differentiated force remain skew-adjoint. The bounded
zero-order term is the sole effect of the weight change on this identity.
Because of \eqref{eq:vp-h-commute}, the Galerkin projection disappears from every
energy pairing. With $R=\E_s(u)^{1/2}$,
\begin{align}\label{eq:vp-h-energyidentity}
 \tfrac12D\E_s(u)[F_{K,N}(u)]
 ={}&\sum_{G\in\{I,A^s,\HH^s\}}\re\ip{Gu}{\LL_u Gu}\nonumber\\
 &+\re\ip{A^s u}{[A^s,E]Tu}\nonumber\\
 &+\re\ip{\HH^s u}{-[\HH^s,v]\partial_xu+E[\HH^s,T]u}.
\end{align}
The first line is bounded in absolute value by $C R\E_s(u)$, using
\eqref{eq:vp-h-density} and \eqref{eq:vp-h-frozen}.

For the spatial commutator, Leibniz' rule expands $[A^s,E]Tu$ into terms
\[
 c_{aj}(\partial_x^jE)\,\partial_x^{a-j}Tu,
 \qquad a\le2s,\quad 1\le j\le a.
\]
By \eqref{eq:vp-h-fieldhigher}, $\norm{\partial_x^jE}_\infty\le C_sR$ for
$j\le2s$. Each derivative on $u$ has total order at most $2s$ because
$a-j+1\le2s$. Lemma~\ref{lem:vp-h-mixed} therefore gives
\begin{equation}\label{eq:vp-h-xcomm}
 \norm{[A^s,E]Tu}\le C_sR^2.
\end{equation}

The velocity commutators are explicit:
\begin{equation}\label{eq:vp-h-hcomm}
 [\HH,v]=-2\partial_v,\qquad
 [\HH,\partial_v]=-2v,\qquad
 [\HH,T]=-2v+b''+2b'\partial_v.
\end{equation}
For $s=1$, Lemma~\ref{lem:vp-h-mixed} directly controls
$[\HH,v]\partial_xu$ and $[\HH,T]u$ by $C_1R$.
For $s=2$, use $[\HH^2,B]=\HH[\HH,B]+[\HH,B]\HH$.
Then $[\HH^2,v]\partial_x$ consists of words of total order at most four,
and $[\HH^2,T]$ consists of words in $v,\partial_v$ of order at most three,
with bounded coefficients made from $b',\ldots,b^{(4)}$ and constants.
For example $[\HH^2,v]=-2(\HH\partial_v+\partial_v\HH)$.
After expanding derivatives of coefficients, Lemma~\ref{lem:vp-h-mixed} yields
\begin{equation}\label{eq:vp-h-vcommbound}
 \norm{[\HH^s,v]\partial_xu}+\norm{[\HH^s,T]u}\le C_sR.
\end{equation}
Insert \eqref{eq:vp-h-xcomm}--\eqref{eq:vp-h-vcommbound} into
\eqref{eq:vp-h-energyidentity}. The result is
$D\E_s\le C_s(R^2+R^3)$, as asserted.
This holds for every state in the energy ball, not only a selected trajectory.
\end{proof}

\subsubsection{Third principle: one-sided stability}
\begin{theorem}\label{thm:vp-h-onesided}
If $\E_s(u),\E_s(\widetilde u)\le Q^2$ for either $s=1$ or $s=2$, then
\begin{equation}\label{eq:vp-one-sided}
 \re\ip{u-\widetilde u}{F_{K,N}(u)-F_{K,N}(\widetilde u)}
 \le\mu_Q\norm{u-\widetilde u}^2,\qquad
 \mu_Q=2c_EQ.
\end{equation}
\end{theorem}
\begin{proof}
Put $z=u-\widetilde u$. Linearity of the field and orthogonality of $\PP$ give
\[
 \re\ip z{F_{K,N}(u)-F_{K,N}(\widetilde u)}
 =\re\ip z{\LL_u z}+\re\ip z{E[z]T\widetilde u}.
\]
The first term is at most $c_EQ\norm z^2/2$.
Since $\HH\ge2I$ and $\norm{\partial_v a}\le\norm{\HH^{1/2}a}$,
we have $\norm{T\widetilde u}\le3Q/2$ for both energies.
Equation~\eqref{eq:vp-h-density} bounds the second term by
$3c_EQ\norm z^2/2$.
\end{proof}
The velocity weight is essential here: integration in $v$ is not a bounded
functional on unweighted $L^2(\R)$, so simply taking an unweighted Hermite
method would leave a gap in this particular one-sided estimate.

\Needspace{5\baselineskip}
\subsection{Mesh-robust first-order error}\label{sec:vp-first-order}
\subsubsection{Time regularity of the continuous solution}
The second-order energy $\E_1$ already suffices for all three principles.
To compare directly with the continuous VP solution, we use $\E_2$ to
control its second time derivative.
Write $M_j(z)=\E_j(z)^{1/2}$ and distinguish the continuous vector field
\[
 F(u)=\LL_u u=-vu_x+E[u]Tu
\]
from its Galerkin restriction $F_{K,N}(u)=\PP F(u)$ on $S_{K,N}$.

\begin{lemma}[Continuous time regularity]\label{lem:vp-consistency}
Let $f$ be a sufficiently regular solution of \eqref{eq:vp-continuous} on $[0,T]$,
set $g=wf$, and assume
\begin{equation}\label{eq:vp-h-continuous-localization}
 \sup_{0\le t\le T}\E_2(g(t))\le Q^2.
\end{equation}
Then
\begin{equation}\label{eq:vp-h-consistency}
 \sup_{0\le t\le T}\norm{g''(t)}\le C_Q.
\end{equation}
\end{lemma}
\begin{proof}
Lemma~\ref{lem:vp-h-mixed}, Leibniz' rule and \eqref{eq:vp-h-fieldhigher} give
\begin{equation}\label{eq:vp-h-mapping}
 \begin{aligned}
 M_1(F(u))&\le C(1+M_2(u))M_2(u),\\
 \norm{F'(u)z}&\le C(1+M_1(u))M_1(z).
 \end{aligned}
\end{equation}
The prime denotes the Fr\'echet derivative with respect to the state.
In the first estimate, $A(vu_x)$ and $\HH(vu_x)$ have total order at most
four, while $A(ETu)$ and $\HH(ETu)$ have order at most three on $u$ and
at most two spatial derivatives on $E$. For the second estimate, linearity
of the field gives
\[
 F'(u)z=-vz_x+E[u]Tz+E[z]Tu.
\]
The mixed term $vz_x$ is controlled by $M_1(z)$ and the remaining terms
by \eqref{eq:vp-h-density}. Since $g'=F(g)$, the chain rule now yields
\[
 \norm{g''}=\norm{F'(g)F(g)}
 \le C(1+M_1(g))M_1(F(g))
 \le C(1+Q)^2Q.
\]
\end{proof}
The continuous energy calculation is the proof of
Theorem~\ref{thm:vp-h-growth} with the projection omitted. On the interval of
existence of a sufficiently regular solution, an initial margin
$\E_2(g_0)=q_0<Q^2$ and $q_0e^{c_{2,Q}T}<Q^2$ therefore ensure
\eqref{eq:vp-h-continuous-localization} by a standard continuation argument.
This supplies a local-in-time justification of the continuous energy
assumption.

\subsubsection{Robust first-order error estimate}
Identify a coefficient matrix with its function in $S_{K,N}$ and consider
\[
 \widehat Y^{n+1}-Y^n-\tau_nF_{K,N}(\widehat Y^{n+1})=r_{n+1},\qquad
 Y^{n+1}=\Tr_{\theta_n}\widehat Y^{n+1}.
\]
Define $g_{\mathrm{SAT}}^n=Y^n$ and the reconstructed distribution
$f_{\mathrm{SAT}}^n=g_{\mathrm{SAT}}^n/w$.
Residual and truncation norms are the coefficient Frobenius norm,
equivalently $\norm{w\,\cdot}$ for distributions.

\begin{theorem}[Robust error for the continuous VP equation]\label{thm:vp-first-order}
Suppose the continuous solution satisfies the hypotheses of
Lemma~\ref{lem:vp-consistency}, and all accepted implicit stages satisfy
$\E_2(\widehat Y^{n+1})\le Q^2$. Assume
$\norm{r_{n+1}}\le\varepsilon_n$,
$\norm{Y^{n+1}-\widehat Y^{n+1}}\le\theta_n$, and
$\mu_Q\tau_n\le\zeta<1$. Set
\[
 \tau=\max_n\tau_n,\qquad
 \delta_{K,N}=(1+(K+1)^2)^{-1}+(2N+2)^{-1}.
\]
Then
\begin{equation}\label{eq:vp-h-pdeerror}
 \begin{aligned}
 \max_{t_n\le T}\norm{f_{\mathrm{SAT}}^n-f(t_n)}
 &\le\max_{t_n\le T}\norm{w(f_{\mathrm{SAT}}^n-f(t_n))}\\
 &\le C_{Q,T,\zeta}\left[
 \norm{Y^0-\PP(wf_0)}+\tau+\delta_{K,N}
 +\sum_n(\varepsilon_n+\theta_n)\right].
 \end{aligned}
\end{equation}
The constant and the stability step restriction are independent of
$K,N$, the retained ranks and small singular values.
\end{theorem}
\begin{proof}
Since $w\ge1$ and $g=wf$,
\[
 \norm{f_{\mathrm{SAT}}^n-f(t_n)}
 \le\norm{w(f_{\mathrm{SAT}}^n-f(t_n))}
 =\norm{g_{\mathrm{SAT}}^n-g(t_n)}.
\]
Decompose the last error as
\[
 \norm{g_{\mathrm{SAT}}^n-g(t_n)}
 \le\norm{Y^n-\PP g(t_n)}+\norm{(I-\PP)g(t_n)}.
\]
The diagonal spectral weights give the projection estimates
\begin{equation}\label{eq:vp-h-projectiontail}
 M_1((I-\PP)g)\le C\delta_{K,N}M_2(g),\qquad
 \norm{(I-\PP)g}\le C\delta_{K,N}M_2(g).
\end{equation}
Indeed, compare $1+a^2+\lambda^2$ with $1+a^4+\lambda^4$ on the
discarded modes, where $a=1+k^2$ and $\lambda=2n+2$. A discarded mode
has either $a\ge1+(K+1)^2$ or $\lambda\ge2N+2$.
Thus the projection error is at most $C_Q\delta_{K,N}$.

For the remaining error, take $X^n=\PP g(t_n)$ as the comparator. By
\eqref{eq:vp-h-commute} and the continuous energy bound,
\begin{equation}\label{eq:vp-h-comparator-energy}
 \E_2(\PP g(t))\le\E_2(g(t))\le Q^2.
\end{equation}
Theorem~\ref{thm:linearstability}, with the one-sided bound
\eqref{eq:vp-one-sided}, therefore gives
\[
 \max_{t_n\le T}\norm{Y^n-\PP g(t_n)}
 \le \exp\!\left(\frac{\mu_QT}{1-\zeta}\right)
 \left[\norm{Y^0-\PP g_0}+
       \sum_n\bigl(\norm{d_{n+1}}+\varepsilon_n+\theta_n\bigr)\right].
\]
It remains to bound the accumulated comparator defects. Since $g'=F(g)$,
\begin{align}\label{eq:vp-h-comparator-defect}
 d_{n+1}
 :={}&X^{n+1}-X^n-\tau_nF_{K,N}(X^{n+1})\nonumber\\
 ={}&\PP\{g(t_{n+1})-g(t_n)-\tau_ng'(t_{n+1})\}
       +\tau_nR_{K,N}(t_{n+1})\nonumber\\
 ={}&-\PP\int_{t_n}^{t_{n+1}}(s-t_n)g''(s)\,ds
       +\tau_nR_{K,N}(t_{n+1}),
\end{align}
where
\begin{equation}\label{eq:vp-h-spatial-residual}
 R_{K,N}(t):=\PP F(g(t))-F_{K,N}(\PP g(t))
            =\PP\{F(g(t))-F(\PP g(t))\}.
\end{equation}
The second bound in \eqref{eq:vp-h-mapping}, applied along the segment between
$g$ and $\PP g$, and \eqref{eq:vp-h-projectiontail} imply
\[
 \norm{R_{K,N}(t)}
 \le C_Q M_1((I-\PP)g(t))\le C_Q\delta_{K,N}.
\]
All states on this segment have $M_1$ bounded by $Q$, since projection
contracts both graph norms and $M_1\le M_2$.
Combining this estimate with Lemma~\ref{lem:vp-consistency} yields
\begin{equation}\label{eq:vp-h-defect-sum}
 \begin{aligned}
 \norm{d_{n+1}}&\le C_Q(\tau_n^2+\tau_n\delta_{K,N}),\\
 \sum_n\norm{d_{n+1}}&\le C_QT(\tau+\delta_{K,N}).
 \end{aligned}
\end{equation}
Substitution into the stability estimate, followed by the projection
bound and the error decomposition at the start of the proof, proves
\eqref{eq:vp-h-pdeerror}.
\end{proof}
The three principles proved above and SVD contraction allow
Theorem~\ref{thm:physical} to provide admissible stages on the common interval
specified by \eqref{eq:uniformsteps}--\eqref{eq:uniformtime}, under the stated energy-residual conditions.
Thus $\varepsilon_n+\theta_n=O(\tau_n^2)$ yields a total error of
order $O(\tau+\delta_{K,N})$, in addition to the initial error.

\section{Numerical experiments}\label{sec:numerics}
We first formulate an energy-constrained implicit SAT algorithm and relate its acceptance conditions to the error analysis. The experiments have three purposes: to verify mesh robustness and first-order convergence, to demonstrate the role of energy constraints in selecting the implicit solution, and to illustrate the step-size restriction when advancing into a larger energy region. All results are computed with the accompanying Python implementation.

\subsection{Energy-constrained implicit SAT}\label{sec:energy-algorithm}\label{sec:numerical-implementation}
We combine energy acceptance with the one-sided stability step restriction from Section~\ref{sec:robust-error}.
Fix an energy radius $Q$ with $\E_h(Y^0)<Q^2$, and a mesh-independent constant $\mu_Q\ge0$ satisfying \eqref{eq:onesided} on $\overline{\Omega_Q}$.
For a fixed $0<\zeta<1$, set
\[
 \tau_{\rm stab}=\begin{cases}\zeta/\mu_Q,&\mu_Q>0,\\+\infty,&\mu_Q=0.\end{cases}
\]
The radius and step restriction are coupled: increasing $Q$ requires a valid stability constant for the enlarged ball and a corresponding update of $\tau_{\rm stab}$.
For a trial step $\tau$ from $(t_n,Y^n)$, define the original backward-Euler residual
\begin{equation}\label{eq:algorithm-residual}
 H_{n,\tau}(Z)=Z-Y^n-\tau F_h(t_n+\tau,Z).
\end{equation}
Let $\eps(\tau)$ and $\theta(\tau)$ be the tolerances for the discrete $L^2$ residual and final SVD tail, respectively.
For first-order convergence, a standard choice is
\begin{equation}\label{eq:algorithm-tolerances}
 \eps(\tau)=C_r\tau^2,\qquad
 \theta(\tau)=C_\theta\tau^2,
\end{equation}
with constants independent of the mesh.

Algorithm~\ref{alg:energy-sat} separates the time-step controller from the implicit solver.
The routine \textsc{EnergyStage} attempts to solve $H_{n,\tau}(Z)=0$, initialized at $Y^n$, using low-rank factors; it returns a candidate and a success flag, or reports failure within its iteration limit.
Algorithm~\ref{alg:energy-newton} implements \textsc{EnergyStage} by low-rank Newton--Krylov iteration with energy-constrained backtracking.

\begin{algorithm}[t]
\caption{Energy-constrained implicit SAT with step rejection}\label{alg:energy-sat}
\begin{algorithmic}[1]
\Require $Y^0$, $T>0$, $Q$, $\tau_{\rm stab}$, trial step $\tau_{\rm try}>0$;
\Statex \hspace{\algorithmicindent} tolerances $\eps,\theta$; reduction factor $0<\beta<1$; retry limit $R$.
\State $t_0\gets0$, $n\gets0$
\While{$t_n<T$}
  \State $\tau\gets\min\{\tau_{\rm try},T-t_n\}$; $\mathrm{accepted}\gets\mathrm{false}$
  \While{$\tau>\tau_{\rm stab}$}
    \State $\tau\gets\beta\tau$
  \EndWhile
  \For{$k=0,\ldots,R$}
    \State $(Z,\mathrm{success})\gets\Call{EnergyStage}{t_n,Y^n,\tau,Q,\eps(\tau)}$
    \If{$\mathrm{success}$}
      \State Recompute $r\gets H_{n,\tau}(Z)$ from the original equation.
      \If{$\E_h(Z)\le Q^2$ and $\|r\|_h\le\eps(\tau)$}
        \State $\mathrm{accepted}\gets\mathrm{true}$; \textbf{break}
      \EndIf
    \EndIf
    \State $\tau\gets\beta\tau$ \Comment{Retry from the same $(t_n,Y^n)$.}
  \EndFor
  \If{not $\mathrm{accepted}$}
    \State \Return the accepted trajectory and a failure flag at $t_n$.
  \EndIf
  \State $\tau_n\gets\tau$, $\widehat Y^{n+1}\gets Z$, $\theta_n\gets\theta(\tau_n)$
  \State $Y^{n+1}\gets\Tr_{\theta_n}\widehat Y^{n+1}$ \Comment{SVD tail in $\|\cdot\|_h$.}
  \State $t_{n+1}\gets t_n+\tau_n$, $\tau_{\rm try}\gets\tau_n$, $n\gets n+1$
\EndWhile
\State \Return the accepted trajectory and a success flag.
\end{algorithmic}
\end{algorithm}

\subsubsection{Error guarantee}\label{sec:algorithm-error}
Consider any accepted prefix $0=t_0<\cdots<t_N\le T$ of Algorithm~\ref{alg:energy-sat}.
Its stages satisfy $\|r_{n+1}\|_h\le\eps(\tau_n)$ and lie in $\overline{\Omega_Q}$ by acceptance.
Lemma~\ref{lem:svd} keeps the truncated states in the same ball, and the step controller ensures $\tau_n\mu_Q\le\zeta<1$.
If the comparison states $X^n$ also lie in $\overline{\Omega_Q}$, Theorem~\ref{thm:linearstability} therefore gives
\begin{equation}\label{eq:algorithm-error}
 \max_{0\le n\le N}\|Y^n-X^n\|_h
 \le e^{\mu_Qt_N/(1-\zeta)}
 \left[e_0+\sum_{n=0}^{N-1}
 \bigl(\|d_{n+1}\|_h+\eps(\tau_n)+\theta(\tau_n)\bigr)\right],
\end{equation}
Here $e_0=\|Y^0-X^0\|_h$, and $d_{n+1}$ is defined in Section~\ref{sec:sat}.
For the parabolic problem, Lemma~\ref{lem:cubic-time-consistency} and \eqref{eq:algorithm-tolerances} give $O(e_0+\max_n\tau_n)$ error relative to the semidiscrete solution. For Vlasov--Poisson, the projected continuous solution has the defect budget \eqref{eq:vp-h-defect-sum}; Theorem~\ref{thm:vp-first-order} gives $O(e_0+\max_n\tau_n+\delta_{K,N})$ total error. Both estimates are uniform in the spatial resolution and ranks.

\subsubsection{Selecting energy-admissible solutions by damping}\label{sec:energy-damping}
We use an inexact Newton--Krylov method in matrix factors; details of the underlying low-rank iteration and arithmetic can be found in \cite[Appendix~A]{RodgersVenturi2023}.
The stages, Newton corrections and Krylov vectors are represented as $USV^\top$.
Linear actions, additions, energy evaluation and SVD recompression can be performed directly on these factors.
For an SVD $Z=\sum_{i=1}^r\sigma_i u_i v_i^\top$, with $\sigma_1\ge\cdots\ge\sigma_r>0$, define
\[
 \mathcal C_\delta(Z)=\sum_{i=1}^{r_\delta}\sigma_i u_i v_i^\top,\qquad
 r_\delta=\min\left\{q\in\{0,\ldots,r\}:\sum_{i=q+1}^r\sigma_i^2\le\delta^2\right\}.
\]
Thus $\|Z-\mathcal C_\delta(Z)\|_F\le\delta$; the empty sum gives the zero matrix.
For a factored input, the SVD is computed through thin QR factorizations and the resulting small core.

Algorithm~\ref{alg:energy-newton} incorporates the energy constraint into damping.
Starting from $Y^n$, each iteration computes a low-rank Newton correction and tries successively smaller damping factors.
A recompressed trial is accepted only if it lies in the prescribed energy ball and satisfies either the residual decrease test or the stopping tolerance.
If the inner solve or backtracking fails, the outer algorithm reduces the time step and retries from $Y^n$.

\begin{algorithm}[!htbp]
\caption{\textsc{EnergyStage}: low-rank solve with energy-constrained damping}\label{alg:energy-newton}
\begin{algorithmic}[1]
\Require $t_n$, $Y^n$, fixed $\tau$, $Q$, discrete $L^2$ residual tolerance $\eps$;
\Statex \hspace{\algorithmicindent} limits $J_{\max},K_{\max}$, $0<c_A<1$, compression tolerances $\delta_j$.
\State $Z\gets Y^n$; abbreviate $H_n=H_{n,\tau}$ from \eqref{eq:algorithm-residual}.
\For{$j=0,\ldots,J_{\max}$}
  \State $a\gets\|H_n(Z)\|_h$.
  \If{$\E_h(Z)\le Q^2$ and $a\le\eps$}
    \State \Return $(Z,\mathrm{true})$.
  \EndIf
  \If{$j=J_{\max}$}
    \State \Return $(\varnothing,\mathrm{false})$.
  \EndIf
  \State Compute a low-rank inexact Newton correction $D$ at $Z$ by Krylov iteration.
  \If{the correction solve fails}
    \State \Return $(\varnothing,\mathrm{false})$.
  \EndIf
  \State $\mathrm{accepted}\gets\mathrm{false}$.
  \For{$k=0,\ldots,K_{\max}$}
    \State $\alpha\gets2^{-k}$; $Z_{\rm tr}\gets\mathcal C_{\delta_j}(Z+\alpha D)$.
    \State $a_{\rm tr}\gets\|H_n(Z_{\rm tr})\|_h$.
    \If{$\E_h(Z_{\rm tr})\le Q^2$ and $a_{\rm tr}\le\max\{(1-c_A\alpha)a,\eps\}$}
      \State $Z\gets Z_{\rm tr}$; $\mathrm{accepted}\gets\mathrm{true}$; \textbf{break}.
    \EndIf
  \EndFor
  \If{not $\mathrm{accepted}$}
    \State \Return $(\varnothing,\mathrm{false})$.
  \EndIf
\EndFor
\end{algorithmic}
\end{algorithm}

\subsubsection{Experimental parameters}
We take $\zeta=0.9$, $\beta=1/2$ and $R=12$.
For the reaction--diffusion tests, $Q=1$ except in the growth experiment of Section~\ref{sec:energy-budget-growth}, where $Q=6$; the Vlasov--Poisson test uses $Q=20$, as specified in Section~\ref{sec:numerical-vp}.
For $\kappa=0.02$, we use the mesh-independent bounds
\[
 \|Z\|_{\ell^\infty}^2\le C_{\rm emb}^2\E_h(Z),\qquad
 C_{\rm emb}^2=\begin{cases}9,&\text{Dirichlet grids, }m\ge64,\\8,&\text{periodic grids, }m\ge32.\end{cases}
\]
Both reaction--diffusion nonlinearities satisfy $f'(s)\le3s^2$.
Diffusion is dissipative, so $\mu_Q=3C_{\rm emb}^2Q^2$ is a valid one-sided constant: $27Q^2$ for the Dirichlet tests and $24Q^2$ for the periodic tests.
The initial energy norms are below $0.559$ and $0.807$, respectively.
Every accepted step records its stability product, energy, residuals and SVD tail.

\subsection{Mesh robustness and first-order convergence}\label{sec:numerical-convergence}
We test the predicted first-order time convergence for reaction--diffusion and Vlasov--Poisson equations. The reaction--diffusion test also refines the spatial grid at fixed time steps to check that the time-error constant remains bounded under refinement.

\subsubsection{Reaction--diffusion equation}\label{sec:numerical-parabolic}
We use the cubic reaction--diffusion problem \eqref{eq:cubicmodel}: $u_t=\kappa\Delta u+u^3$ on $(0,1)^2$, with zero Dirichlet data, $u_0(x,y)=16x(1-x)y(1-y)$, and $\kappa=1/50$.
Centered differences with $h=(m+1)^{-1}$ give
\[
 X_h'=-B_{x,h}X_h-X_hB_{y,h}+X_h^{\circ3},\qquad
 B_{x,h}=B_{y,h}=\frac{\kappa}{h^2}\operatorname{tridiag}(-1,2,-1).
\]
We use $\|Z\|_h=h\|Z\|_F$ and the graph energy \eqref{eq:graphenergy}, verified in Section~\ref{sec:parabolic}.
The initial sampled matrix has rank one. We use second-order centered differences with $m=64,128,256,512$ interior points per direction and integrate to $T=0.5$. In the graph energy \eqref{eq:graphenergy}, $B_{x,h}=B_{y,h}=-(1/50)\Delta_h^{\mathrm{Dir}}$.
With $X_{h,\mathrm{ref}}$ denoting an independently computed reference trajectory on the same spatial grid, we measure
\begin{equation}\label{eq:measerror}
 E_N=\max_{0\le n\le N}\|Y^n-X_{h,\mathrm{ref}}(t_n)\|_h,
 \qquad p_N=\log_2(E_{N/2}/E_N).
\end{equation}
The step tolerances are
\begin{equation}\label{eq:standardtols}
 \eps_n=10^{-5}\tau^2,\qquad
 \theta_n=10^{-5}\tau^2.
\end{equation}
The reported runs additionally imposed $\|r_{n+1}\|_{\E,h}\le10^{-4}\tau$.

Table~\ref{tab:parabolic} and Figure~\ref{fig:parabolic} show the maximum-in-time errors. The last observed order is about $1.0145$ on all four grids. At $N=256$, the relative spread between the largest and smallest errors across these grids is about $0.188\%$, although the largest diffusion eigenvalue increases by a factor of about 62.3. The retained rank is at most seven in this family. At a fixed time step, the measured errors remain nearly unchanged under mesh refinement.
For the two-dimensional discrete diffusion operator, $\|L_h\|=8(1/50)(m+1)^2\cos^2(\pi/(2(m+1)))$. At the common time step $\tau=0.5/256$, the product $\tau\|L_h\|$ increases from 1.32 to 82.24 across the four grids. The time step is unchanged, while the errors remain almost identical. This directly illustrates the distinction between a step restriction based on the stiff operator norm and one based on the physical energy region.

\begin{table}[t]
\centering\small
\caption{Reaction--diffusion errors $E_N$ at $T=0.5$. Last row: order between the two finest time steps.}\label{tab:parabolic}
\begin{tabular}{rcccc}
\toprule
$N$ & $m=64$ & $m=128$ & $m=256$ & $m=512$ \\
\midrule
16 & $1.768\times10^{-2}$ & $1.765\times10^{-2}$ & $1.765\times10^{-2}$ & $1.764\times10^{-2}$ \\
32 & $8.026\times10^{-3}$ & $8.014\times10^{-3}$ & $8.011\times10^{-3}$ & $8.010\times10^{-3}$ \\
64 & $3.843\times10^{-3}$ & $3.837\times10^{-3}$ & $3.836\times10^{-3}$ & $3.835\times10^{-3}$ \\
128 & $1.882\times10^{-3}$ & $1.879\times10^{-3}$ & $1.879\times10^{-3}$ & $1.879\times10^{-3}$ \\
256 & $9.317\times10^{-4}$ & $9.303\times10^{-4}$ & $9.300\times10^{-4}$ & $9.299\times10^{-4}$ \\
Last order & 1.0145 & 1.0145 & 1.0145 & 1.0145 \\
\bottomrule
\end{tabular}

\end{table}
\begin{figure}[t]
\centering\includegraphics[width=\linewidth]{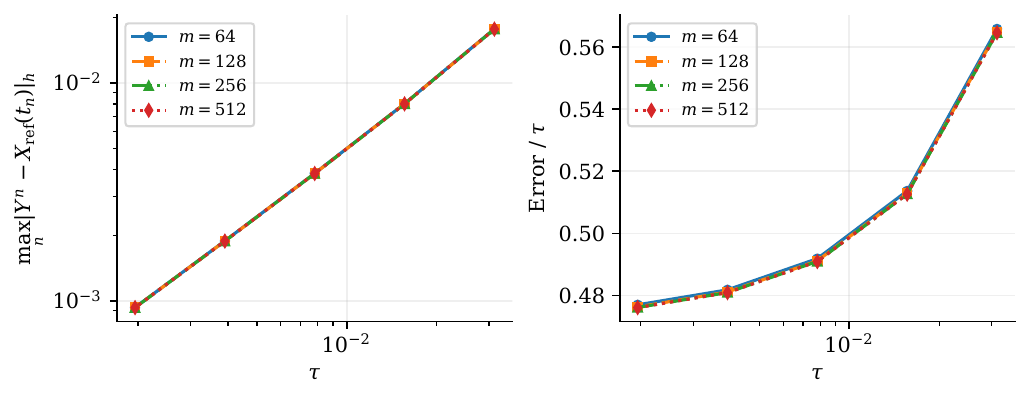}
\caption{Reaction--diffusion time convergence. Left: errors on four spatial grids. Right: error divided by the time step. The almost coincident curves show that the observed time-error constant does not grow with the norm of the diffusion operator.}\label{fig:parabolic}
\end{figure}

\subsubsection{Vlasov--Poisson system}\label{sec:numerical-vp}
We use the weighted Fourier--Hermite discretization of Section~\ref{sec:vp-discretization}
for the original VP equation on $\T\times\R$, with repulsive coupling $\chi=-1$ and
\begin{equation}\label{eq:vp-numerical-initial}
 f_0(x,v)=\frac{e^{-v^2/2}}{\sqrt{2\pi}}(1+0.05\cos x).
\end{equation}
The initial coefficient matrix represents $\PP_{K,N}(wf_0)$.
We fix $(K,N)=(16,64)$, corresponding to real coefficient matrices of size
$33\times64$, and integrate to $T=1$ with 64, 128, 256 and 512 time steps.
The energy is $\E_2$ from \eqref{eq:vp-hermite-energy}, with $Q=20$.
The one-sided constant \eqref{eq:vp-one-sided} is
$\mu_Q=2\pi Q/\sqrt6\approx51.302$; hence all tested steps satisfy
$\tau\mu_Q\le0.802<0.9$.
The residual and final-tail tolerances in the coefficient Frobenius norm are
$\varepsilon_n=\theta_n=10^{-5}\tau_n^2$.
States, Newton corrections and Krylov vectors use real matrix factors.
Fourier products are evaluated without aliasing, and the original residual
is checked independently of factor recompression.

Here $Y^n$ is the Fourier--Hermite coefficient matrix of the approximation to
$g=wf$ at time $t_n$, with $w(v)=\sqrt{1+v^2}$ as in \eqref{eq:vp-h-transform}.
Using real orthonormal Fourier modes $\psi_j$ and the normalized Hermite functions
$\phi_m$ from Section~\ref{sec:vp-discretization}, this correspondence is
\[
 g_{\mathrm{SAT}}^n(x,v)=\sum_{j=0}^{2K}\sum_{m=0}^{N-1}
 Y^n_{jm}\psi_j(x)\phi_m(v),\qquad
 f_{\mathrm{SAT}}^n=g_{\mathrm{SAT}}^n/w.
\]
Let $U_{K,N,\mathrm{ref}}(t)$ denote the reference coefficient matrix for the same
Galerkin ODE, and $f_{K,N,\mathrm{ref}}(t)$ its corresponding distribution.
The errors reported in Table~\ref{tab:vp-hermite-convergence} are
\[
 E_{\rm time}=\max_n\|Y^n-U_{K,N,\mathrm{ref}}(t_n)\|_F.
\]
By orthonormality, these equal the maximum of
$\|w(f_{\mathrm{SAT}}^n-f_{K,N,\mathrm{ref}}(t_n))\|_{L^2}$ over the time nodes.
The last observed order is approximately $0.997$.
Figure~\ref{fig:vp-hermite-convergence} shows first-order decay of the error
as the time step decreases and convergence of the ratio $E_{\rm time}/\tau$.
The largest retained rank is nine, and the computed stage energy norms
remain below $14.74$, within the prescribed ball.

The reported errors remain unchanged at the displayed precision when evaluated
against a reference on a finer spatial discretization, showing that spatial
discretization does not affect the observed convergence.

\begin{table}[t]
\centering\small
\caption{Maximum weighted time errors for the original VP equation
\eqref{eq:vp-numerical-initial}, with $(K,N)=(16,64)$, $T=1$ and $\tau=1/(\text{steps})$.
The reference is independently integrated on the same Galerkin space.}
\label{tab:vp-hermite-convergence}
\begin{tabular}{rccc}
\toprule
Steps & $\tau$ & Maximum weighted error & Order \\
\midrule
64 & $1/64$ & $1.364\times10^{-3}$ & -- \\
128 & $1/128$ & $6.879\times10^{-4}$ & $0.9878$ \\
256 & $1/256$ & $3.454\times10^{-4}$ & $0.9940$ \\
512 & $1/512$ & $1.731\times10^{-4}$ & $0.9970$ \\
\bottomrule
\end{tabular}

\end{table}
\begin{figure}[t]
\centering\includegraphics[width=\linewidth]{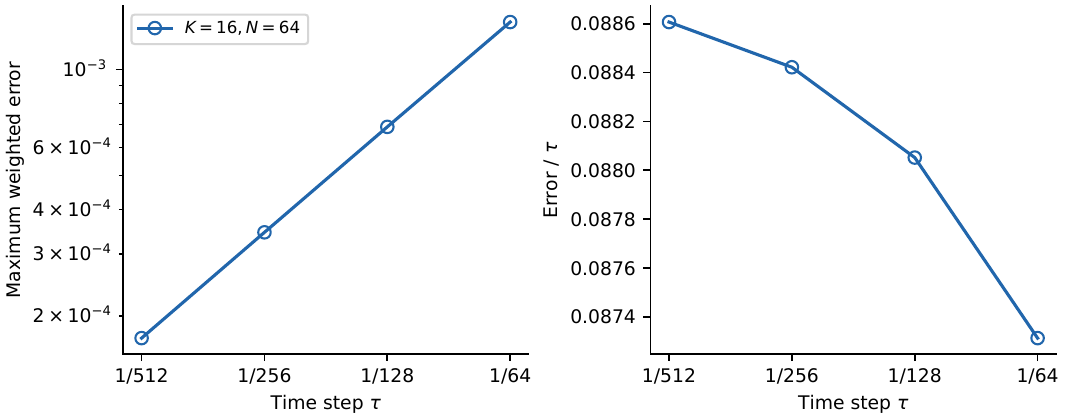}
\caption{Fourier--Hermite VP time convergence with $(K,N)=(16,64)$. Left: maximum weighted time errors.
Right: error divided by the time step.}\label{fig:vp-hermite-convergence}
\end{figure}

\subsection{Existing methods without energy constraints may compute incorrect solutions}\label{sec:numerical-root-selection}
We consider the smooth reaction--diffusion equation
\begin{equation}\label{eq:root-selection-pde}
\begin{aligned}
 u_t&=0.02\Delta u-1000\tanh(10u)+u^3,\qquad (x,y)\in[0,1)^2,\\
 u_0(x,y)&=0.8+0.1\cos(2\pi x)\cos(2\pi y)
                   +0.04\sin(4\pi x)\sin(2\pi y),
\end{aligned}
\end{equation}
with periodic boundary conditions and centered differences on a $64^2$ grid.
The initial matrix has rank three and energy norm below $0.807$.
Since $0.66\le u_0\le0.94$ and the reaction is negative on $(0,1]$, the exact solution remains nonnegative and decays toward zero.

We compare Algorithm~\ref{alg:energy-sat} with a matrix-factor implementation of the unconstrained implicit SAT strategy of \cite{RodgersVenturi2023}.
Both use the same Newton--Krylov solver with residual backtracking, initialized at the preceding SAT state.
The energy-constrained algorithm additionally checks the energy of each damped trial.
Both methods take four steps of length $\tau=0.025$ to $T=0.1$, with residual tolerance $10^{-11}$ and SVD-tail tolerance $10^{-3}\tau^2$.
For $Q=1$, the bound $\mu_Q=24$ gives $\tau\mu_Q=0.6<0.9$; no step reduction is needed.
Set $M=0.94$ and $\alpha=1000\tanh(10M)/M-M^2$. Since $\tanh(10s)/s$ decreases for $s>0$, the reaction satisfies $f(s)\le-\alpha s$ on $[0,M]$. The initial values lie in this interval, and periodic centered differences satisfy the discrete comparison principle. Comparison with the spatially constant supersolution $Me^{-\alpha t}$, together with the unit area of the domain, gives
\[
 0\le X_h(t)\le Me^{-\alpha t}\quad\hbox{entrywise},\qquad
 \|X_h(t)\|_h\le Me^{-\alpha t}.
\]
Here $\alpha\approx1062.946$. At the first positive output time $t=0.025$, this bound is $2.7061\times10^{-12}<3\times10^{-12}$ and decreases thereafter, so zero provides a reference to this accuracy at the positive output times.

On the $64^2$ grid, the unconstrained solve accepts a first stage of energy norm $11.4184$ despite a residual below $10^{-11}$.
Energy-constrained damping instead selects a stage of energy norm $0.003214$.
Figure~\ref{fig:root-selection} compares the reference solution and the two SAT approximations through the time trace at $(x,y)=(0,0)$ and the spatial slice at $y=0$ after the first step.
The unconstrained trajectory becomes negative and follows an exterior branch, whereas Algorithm~\ref{alg:energy-sat} recovers the decay toward zero.
The maximum errors at the common time nodes are $11.4184$ and $3.1956\times10^{-3}$, respectively.

\begin{figure}[t]
\centering\includegraphics[width=\linewidth]{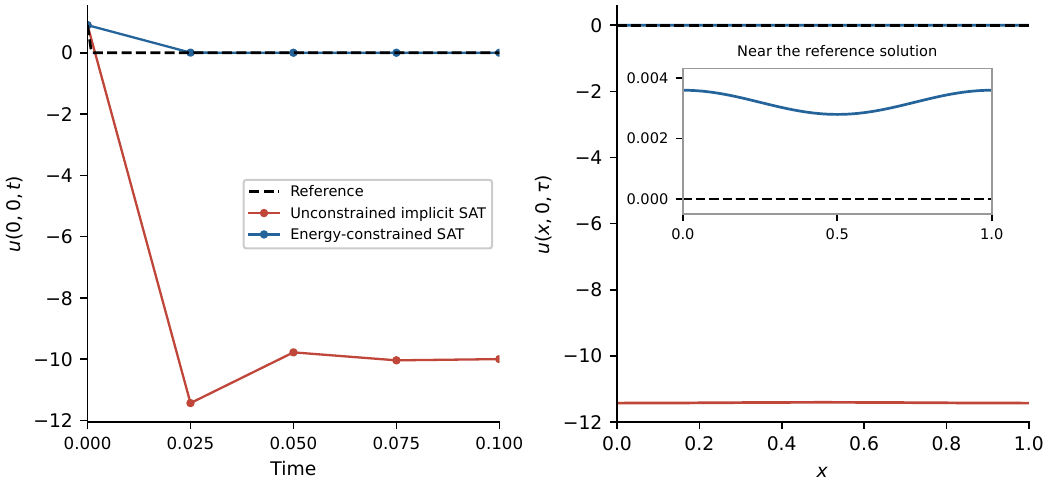}
\caption{Reference solution and the two implicit SAT approximations for \eqref{eq:root-selection-pde} on the $64^2$ grid, with $\tau=0.025$. Left: $u(0,0,t)$, with markers at the SAT time nodes. Right: $u(x,0,\tau)$; the inset resolves the energy-constrained approximation and the reference near zero. Both methods attain the residual tolerance, but the unconstrained method selects an incorrect solution.}\label{fig:root-selection}
\end{figure}

\subsection{Practical meaning of the step restriction: smaller steps for larger energy regions}\label{sec:energy-budget-growth}
To examine continuation into a larger energy region, consider
\[
 u_t=0.02\Delta u+u^3,\qquad u_0(x,y)=0.8,
\]
on the same periodic domain, using a $32^2$ grid.
The solution is spatially constant,
\[
 u(t,x,y)=a(t)=\frac{0.8}{\sqrt{1-1.28t}},\qquad 0\le t\le T=0.75.
\]
Its energy norm equals $a(t)$ and reaches $4$ at $T$, so the initial budget $Q=1$ must be enlarged; we choose $Q=6$.

Both runs require the implicit stage to satisfy $\E_h(\widehat Y^{n+1})^{1/2}\le6$ and use the same residual and truncation tolerances as above, but choose different time steps.
The first keeps $\tau=0.025$, the step selected for $Q=1$.
The second updates the stability restriction to $\mu_6=24\cdot6^2=864$ and uses Algorithm~\ref{alg:energy-sat}, which accepts $\tau=0.00078125$ with $\tau\mu_6=0.675$.

The constant spatial mode is invariant under both the semidiscrete equation and the factor iteration, and all accepted states have rank one.
Writing $Y^n=a_n\boldsymbol1\boldsymbol1^\top$, the stage equation reduces to
\[
 a-\tau a^3=a_n,\qquad |a|\le6.
\]
The frozen-step run reaches $a_n=3.22020$ at $t_n=0.675$ and fails at the next step.
Here the absence of an admissible constant stage can be checked directly:
\[
 \max_{|a|\le6}(a-0.025a^3)
 =\frac{2}{3\sqrt{3\cdot0.025}}\approx2.43432<3.22020.
\]
Its only real constant root is approximately $-7.55335$, outside the ball.
The updated-step run reaches $T=0.75$ with energy norm $4.12954<6$.
Table~\ref{tab:energy-budget-growth} and Figure~\ref{fig:energy-budget-growth} show that enlarging the energy budget while retaining its former step can obstruct continuation; updating the step restriction allows advancement to the prescribed final time.

\begin{table}[t]
\centering\small
\caption{Continuation with acceptance radius $Q=6$. The frozen-step control retains the step selected for $Q=1$; the updated run uses the stability bound for $Q=6$.}\label{tab:energy-budget-growth}
\begin{tabular}{lcccc}
\toprule
Step selection & $\tau$ & $\tau\mu_6$ & Last accepted time & Energy norm \\
\midrule
Frozen $Q=1$ step & $0.025$ & $21.6$ & $0.675$ & $3.22020$ \\
Updated $Q=6$ step & $0.00078125$ & $0.675$ & $0.750$ & $4.12954$ \\
\bottomrule
\end{tabular}

\end{table}
\begin{figure}[t]
\centering\includegraphics[width=0.82\linewidth]{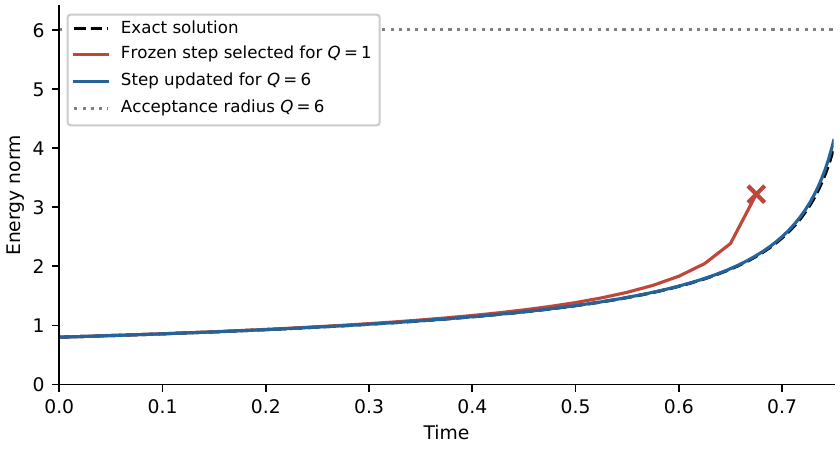}
\caption{Growth of the spatially constant solution. The cross marks the last accepted frozen-step state. Updating the step for $Q=6$ permits continuation to $T=0.75$.}\label{fig:energy-budget-growth}
\end{figure}

\FloatBarrier
\section{Conclusions}\label{sec:conclusion}
We have developed a unified framework for mesh-robust error analysis of implicit SAT methods. Existing low-rank error bounds may involve full Lipschitz constants that grow under spatial refinement, limiting their usefulness, particularly for kinetic equations. Our framework provides a discrete counterpart of energy localization for nonlinear PDEs and addresses this open problem for the class of implicit SAT schemes covered by our assumptions.

The analysis combines three mesh-uniform conditions: an initial energy bound, controlled energy propagation, and one-sided Lipschitz stability within the energy-localized region. Ordinary SVD truncation preserves the required bounds for energies formed from separate left and right operators. Under the associated step-size and residual conditions, the implicit stage exists and is unique in the prescribed energy ball for each admissible residual. The resulting error estimates and time-step restrictions are independent of spatial resolution and require no lower bound on the numerical singular values.

For semilinear parabolic and Vlasov--Poisson equations, we establish robust first-order error bounds for backward-Euler SAT under suitable regularity assumptions. The parabolic analysis applies to the full implicit equation without separating the stiff part. For the original Vlasov--Poisson system on $\T\times\R$, the total error is of order $O(\tau+\delta_{K,N})$, where $\delta_{K,N}$ measures spatial approximation, with constants uniform in both spectral cutoffs.

We also design an energy-constrained implicit SAT algorithm that combines energy acceptance with a one-sided stability step restriction. Numerical experiments support the predicted first-order convergence and demonstrate the practical role of energy localization. Without energy constraints, implicit SAT can select incorrect solution branches even when algebraic residuals are small. Moreover, continuation into larger energy regions can require smaller time steps. The energy constraint and the corresponding step restriction therefore guide both implicit solution selection and continuation.

\section*{Acknowledgements}
This work was supported by the National Natural Science Foundation of China under Grant Nos. 12494543 and 92370205, and by the 
Strategic Priority Research Program of the Chinese Academy of Sciences under Grant No. XDA0480504. 
The computations were partially performed on the LSSC4 PC cluster of the State Key Laboratory of Scientific and Engineering Computing.

\appendix
\section{Time consistency for semilinear reaction--diffusion equations}\label{app:cubic-time-consistency}
We prove the integral bound in Lemma~\ref{lem:cubic-time-consistency}.
Throughout, the hypotheses of Section~\ref{sec:parabolic} hold and all constants are independent of $h$.
Fractional powers of the self-adjoint positive definite matrix operator $L_h$ are defined spectrally on its orthonormal eigenmatrices:
if $L_h\psi_\nu=\Lambda_\nu\psi_\nu$, then
$L_h^\gamma Z=\sum_\nu\Lambda_\nu^\gamma\ip{Z}{\psi_\nu}_h\psi_\nu$ for $\gamma>0$.

\emph{Initial regularity.}
Set $\beta=\sigma/2\in(0,1/4)$.
The eigenvalues are $\Lambda_{k,h}=\lambda_h(k_1)+\lambda_h(k_2)$ in the Fourier representation used in Section~\ref{sec:parabolic-principles}.
Their upper bound and the sampling estimate \eqref{eq:initial-sampling}, with $s=2+\sigma$, give
\[
 \|L_h^{1+\beta}X_{h,0}\|_h^2
 \le C_{\kappa,\sigma}
 \sum_{k\in K_n}(1+|k|^2)^{2+\sigma}|\widehat x_{h,0}(k)|^2
 \le C_{\kappa,\sigma}\mathcal A_{2+\sigma}(u_0)^2.
\]
For $0<\beta\le1$, $\lambda^{2\beta}\le1+\lambda^2$ implies
\[
 \|L_h^\beta Z\|_h^2
 \le\|Z\|_h^2+\|L_hZ\|_h^2\le2\E_h(Z).
\]
Consequently, the composition bound \eqref{eq:cubicsobolev} yields
\begin{equation}\label{eq:cubicinitialregularity}
 \|L_h^\beta F_h(X_{h,0})\|_h
 \le \|L_h^{1+\beta}X_{h,0}\|_h
      +\sqrt2 K_Q\|X_{h,0}\|_{\E,h}
 \le C_{Q,\sigma,u_0}.
\end{equation}

\emph{An integrable second time derivative.}
On the common energy interval, set
$W_h=X_h'$ and $N_h=(f')^\circ(X_h)\circ W_h$.
Differentiating the finite-dimensional equation gives
\[
 W_h'=-L_hW_h+N_h,\qquad W_h(0)=F_h(X_{h,0}).
\]
The energy bound directly controls the basic norms:
\[
 \begin{aligned}
 \|W_h(t)\|_h
 &\le\|L_hX_h(t)\|_h+\|f^\circ(X_h(t))\|_h
 \le(\sqrt2+M_{1,Q})Q,\\
 \|N_h(t)\|_h&\le M_{1,Q}(\sqrt2+M_{1,Q})Q.
 \end{aligned}
\]
Write
\[
 W_h(t)=e^{-tL_h}W_h(0)+V_h(t),\qquad
 V_h'+L_hV_h=N_h,\qquad V_h(0)=0.
\]
By the spectral theorem,
$\|L_h^{1-\beta}e^{-tL_h}\|\le C_\beta t^{\beta-1}$.
Thus \eqref{eq:cubicinitialregularity} gives
\[
 \int_0^T\left\|\frac{\dd}{\dd t}e^{-tL_h}W_h(0)\right\|_h\dd t
 \le C_\beta\|L_h^\beta W_h(0)\|_h
       \int_0^T t^{\beta-1}\dd t
 \le C_{Q,T,\sigma,u_0}.
\]
For the forced part, square $V_h'+L_hV_h=N_h$ in the physical norm and integrate:
\[
 \int_0^T\|V_h'\|_h^2\dd t
 +\int_0^T\|L_hV_h\|_h^2\dd t
 +\|L_h^{1/2}V_h(T)\|_h^2
 =\int_0^T\|N_h\|_h^2\dd t.
\]
The cross term is the time derivative of $\|L_h^{1/2}V_h\|_h^2$, and $V_h(0)=0$ removes the initial contribution.
Cauchy--Schwarz therefore gives
\[
 \int_0^T\|V_h'(t)\|_h\dd t
 \le \sqrt T\left(\int_0^T\|N_h(t)\|_h^2\dd t\right)^{1/2}
 \le T\sup_{0\le t\le T}\|N_h(t)\|_h.
\]
Since $X_h''=W_h'$, these two bounds prove the first estimate in \eqref{eq:cubicdefect}.

\bibliographystyle{siamplain}
\bibliography{references}

@article{MeyerMantheCederbaum1990,
 author={Meyer, H.-D. and Manthe, U. and Cederbaum, L. S.},
 title={The multi-configurational time-dependent {Hartree} approach},
 journal={Chem. Phys. Lett.}, volume={165}, number={1}, pages={73--78}, year={1990},
 doi={10.1016/0009-2614(90)87014-I}}

@article{Kormann2015,
 author={Kormann, Katharina},
 title={A semi-{Lagrangian} {Vlasov} solver in tensor train format},
 journal={SIAM J. Sci. Comput.}, volume={37}, number={4}, pages={B613--B632}, year={2015},
 doi={10.1137/140971270}}

@article{PengMcClarrenFrank2020,
 author={Peng, Zhuogang and McClarren, Ryan G. and Frank, Martin},
 title={A low-rank method for two-dimensional time-dependent radiation transport calculations},
 journal={J. Comput. Phys.}, volume={421}, pages={109735}, year={2020},
 doi={10.1016/j.jcp.2020.109735}}

@article{Bachmayr2023,
 author={Bachmayr, Markus},
 title={Low-rank tensor methods for partial differential equations},
 journal={Acta Numer.}, volume={32}, pages={1--121}, year={2023},
 doi={10.1017/S0962492922000125}}

@article{KochLubich2007,
 author={Koch, Othmar and Lubich, Christian},
 title={Dynamical low-rank approximation},
 journal={SIAM J. Matrix Anal. Appl.}, volume={29}, number={2}, pages={434--454}, year={2007},
 doi={10.1137/050639703}}

@article{CerutiLubich2022,
 author={Ceruti, Gianluca and Lubich, Christian},
 title={An unconventional robust integrator for dynamical low-rank approximation},
 journal={BIT Numer. Math.}, volume={62}, pages={23--44}, year={2022},
 doi={10.1007/s10543-021-00873-0}}

@article{CerutiKuschLubich2022,
 author={Ceruti, Gianluca and Kusch, Jonas and Lubich, Christian},
 title={A rank-adaptive robust integrator for dynamical low-rank approximation},
 journal={BIT Numer. Math.}, volume={62}, pages={1149--1174}, year={2022},
 doi={10.1007/s10543-021-00907-7}}

@article{GuoQiu2022,
 author={Guo, Wei and Qiu, Jing-Mei},
 title={A low rank tensor representation of linear transport and nonlinear {Vlasov} solutions and their associated flow maps},
 journal={J. Comput. Phys.}, volume={458}, pages={111089}, year={2022},
 doi={10.1016/j.jcp.2022.111089}}

@article{GuoQiu2024,
 author={Guo, Wei and Qiu, Jing-Mei},
 title={A conservative low rank tensor method for the {Vlasov} dynamics},
 journal={SIAM J. Sci. Comput.}, volume={46}, number={1}, pages={A232--A263}, year={2024},
 doi={10.1137/22M1473960}}

@article{Review2025,
 author={Einkemmer, Lukas and Kormann, Katharina and Kusch, Jonas and McClarren, Ryan G. and Qiu, Jing-Mei},
 title={A review of low-rank methods for time-dependent kinetic simulations},
 journal={J. Comput. Phys.}, volume={538}, pages={114191}, year={2025},
 doi={10.1016/j.jcp.2025.114191}}

@article{LubichOseledets2014,
 author={Lubich, Christian and Oseledets, Ivan V.},
 title={A projector-splitting integrator for dynamical low-rank approximation},
 journal={BIT Numer. Math.}, volume={54}, pages={171--188}, year={2014},
 doi={10.1007/s10543-013-0454-0}}

@article{KieriLubichWalach2016,
 author={Kieri, Emil and Lubich, Christian and Walach, Hanna},
 title={Discretized dynamical low-rank approximation in the presence of small singular values},
 journal={SIAM J. Numer. Anal.}, volume={54}, number={2}, pages={1020--1038}, year={2016},
 doi={10.1137/15M1026791}}

@article{OstermannPiazzolaWalach2019,
 author={Ostermann, Alexander and Piazzola, Chiara and Walach, Hanna},
 title={Convergence of a low-rank {Lie--Trotter} splitting for stiff matrix differential equations},
 journal={SIAM J. Numer. Anal.}, volume={57}, number={4}, pages={1947--1966}, year={2019},
 doi={10.1137/18M1177901}}

@article{RodgersDektorVenturi2022,
 author={Rodgers, Abram and Dektor, Alec and Venturi, Daniele},
 title={Adaptive integration of nonlinear evolution equations on tensor manifolds},
 journal={J. Sci. Comput.}, volume={92}, pages={39}, year={2022},
 doi={10.1007/s10915-022-01868-x}}

@article{RodgersVenturi2023,
 author={Rodgers, Abram and Venturi, Daniele},
 title={Implicit integration of nonlinear evolution equations on tensor manifolds},
 journal={J. Sci. Comput.}, volume={97}, pages={33}, year={2023},
 doi={10.1007/s10915-023-02352-w}}

@article{ElKahzaEtAl2024,
 author={El Kahza, Hamad and Taitano, William and Qiu, Jing-Mei and Chac{\'o}n, Luis},
 title={Krylov-based adaptive-rank implicit time integrators for stiff problems with application to nonlinear {Fokker--Planck} kinetic models},
 journal={J. Comput. Phys.}, volume={518}, pages={113332}, year={2024},
 doi={10.1016/j.jcp.2024.113332}}

@article{AppeloCheng2025,
 author={Appel{\"o}, Daniel and Cheng, Yingda},
 title={Robust implicit adaptive low rank time-stepping methods for matrix differential equations},
 journal={J. Sci. Comput.}, volume={102}, pages={81}, year={2025},
 doi={10.1007/s10915-025-02808-1}}

@article{NakaoQiuEinkemmer2025,
 author={Nakao, Joseph and Qiu, Jing-Mei and Einkemmer, Lukas},
 title={Reduced augmentation implicit low-rank ({RAIL}) integrators for advection-diffusion and {Fokker--Planck} models},
 journal={SIAM J. Sci. Comput.}, volume={47}, number={2}, pages={A1145--A1169}, year={2025},
 doi={10.1137/23M1622921}}

@incollection{Kato1975,
 author={Kato, Tosio},
 title={Quasi-linear equations of evolution, with applications to partial differential equations},
 booktitle={Spectral Theory and Differential Equations},
 editor={Everitt, William N.}, series={Lecture Notes in Mathematics}, volume={448},
 publisher={Springer}, address={Berlin}, pages={25--70}, year={1975},
 doi={10.1007/BFb0067080}}

@book{Taylor2011,
 author={Taylor, Michael E.},
 title={Partial Differential Equations {III}: Nonlinear Equations},
 series={Applied Mathematical Sciences}, volume={117}, edition={second},
 publisher={Springer}, address={New York}, year={2011},
 doi={10.1007/978-1-4419-7049-7}}

@article{EinkemmerLubich2018,
 author={Einkemmer, Lukas and Lubich, Christian},
 title={A low-rank projector-splitting integrator for the {Vlasov--Poisson} equation},
 journal={SIAM J. Sci. Comput.}, volume={40}, number={5}, pages={B1330--B1360}, year={2018},
 doi={10.1137/18M116383X}}

@book{Deimling1985,
 author={Deimling, Klaus},
 title={Nonlinear Functional Analysis},
 publisher={Springer}, address={Berlin}, year={1985},
 doi={10.1007/978-3-662-00547-7}}

@misc{BongioanniTorrea2006,
 author={Bongioanni, B. and Torrea, J. L.},
 title={Sobolev spaces associated to the harmonic oscillator},
 year={2006}, note={arXiv:math/0608684},
 url={https://arxiv.org/abs/math/0608684}}

@article{GreenPatrickSpiteri2019,
 author={Green, Kevin R. and Patrick, George W. and Spiteri, Raymond J.},
 title={On theoretical upper limits for valid timesteps of implicit {ODE} methods},
 journal={AIMS Math.}, volume={4}, number={6}, pages={1841--1853}, year={2019},
 doi={10.3934/math.2019.6.1841}}
\end{document}